\documentclass[a4paper]{amsart}

\usepackage{amsmath}
\usepackage{amssymb}
\usepackage{amsxtra}
\usepackage{color}
\usepackage{float}
\usepackage{pdflscape}
\usepackage{varioref}
\usepackage{colortbl}
\usepackage[bookmarksnumbered,colorlinks,plainpages]{hyperref}
\usepackage{cite}
\usepackage{stmaryrd}

\usepackage{graphicx}
\usepackage{amsthm}
\usepackage[all,2cell]{xy}
\usepackage{amsfonts}
\usepackage{xypic,leqno,amscd,amssymb,pstricks,latexsym,amsbsy,xypic,mathrsfs,verbatim}
\usepackage{multirow}

\newtheorem*{Thm-A}{Theorem A}
\newtheorem*{Thm-B}{Theorem B}

\newtheorem{thm}{Theorem}[section]

\newtheorem{lem}[thm]{Lemma}

\newtheorem{prop}[thm]{Proposition}

\theoremstyle{definition}
\newtheorem{defn}[thm]{Definition}
\newtheorem{rem}[thm]{\bf Remark}
\newtheorem{exm}[thm]{Example}
\numberwithin{equation}{section}

\def\Hom{{\rm Hom}}

\def\Ext{{\rm Ext}}

\def\im {{\rm im\, }}

\def\co{{\mathcal O}}

\def\coh{{\rm coh}\mbox{-}}

\def\bbX{{\mathbb X}}
\def\bbL{{\mathbb L}}

\def\bbZ{{\mathbb Z}}

\def\bfq{{\mathbf{q}}}
\def\bfp{{\mathbf{p}}}
\def\bfla{{{\boldsymbol\lambda}}}
\def\bfmu{{{\boldsymbol\mu}}}

\def\cS{{\mathcal S}}
\def\cT{{\mathcal T}}
\begin{document}
\title[Admissible homomorphisms and equivariant relations for WPL]{Admissible homomorphisms and equivariant relations between weighted projective lines}

\author[J. Chen, Y. Lin, S. Ruan and H. Zhang] {Jianmin Chen, Yanan Lin, Shiquan Ruan and Hongxia Zhang}

\date{\today}
\maketitle

\dedicatory{}%
\commby{}%

\begin{abstract} The string group acts on the category of coherent sheaves over a weighted projective line by degree-shift actions. We study the equivariant equivalence relations induced by degree-shift actions between weighted projective lines. We prove that such an equivariant equivalence is characterized by an admissible homomorphism between the associated string groups. We classify all these equivariant equivalences for the weighted projective lines of domestic and tubular types.
\end{abstract}

\maketitle

\section{Introduction}

Weighted projective lines are introduced in \cite{GL87}, which provide a geometric approach to canonical algebras in the sense of \cite{Rin}.
It is well known that weighted projective lines are related to compact Riemann surfaces or smooth projective curves via group actions; see \cite{Len16}. More precisely, any weighted projective line $\bbX$ is isomorphic to the orbifold quotient $X/G$ for a compact Riemann surface $X$ and a finite subgroup $G$ of the automorphism group ${\rm Aut}(X)$, where the weight structure of $\bbX$ is given by the ramification data of the quotient map $X \to X/G$. In particular, the category of coherent sheaves over $\bbX$ is equivalent to the category of $G$-equivariant coherent sheaves on $X$.

Moreover, by using group actions on the category of coherent sheaves, the authors in \cite{GL87, CCZ} established connections between weighted projective lines of tubular types and smooth elliptic curves, compare \cite{Po}. Recently, the authors found equivariant relations between certain weighted projective lines of tubular types in \cite{CC17} and \cite{CCR}, by considering the degree-shift actions on the category of coherent sheaves. In the present paper we study these equivariant relations systematically for all the weighted projective lines. In particular, we give a complete classification for domestic and tubular types.

Before stating the main results of this work, let us briefly introduce some notation first, for more details we refer to Section 2.

Let ${\bf p}=(p_1, p_2, \cdots, p_t)$ be a sequence consisting of integers with $p_i\geq 2$.
The string group $\bbL({\bf p})$ of type ${\bf p}$ is an abelian group on generators $\vec{x}_1, \vec{x}_2, \cdots, \vec{x}_t$ subject to the relations $p_1\vec{x}_1=p_2\vec{x}_2=\cdots = p_t\vec{x}_t$. A group homomorphism between string groups is called \emph{admissible} if it satisfies certain technical conditions, see Definition \ref{defn:AH}.
Let $\mathbf{k}$ be an algebraically closed field. Let ${\boldsymbol\lambda}=(\lambda_1, \lambda_2, \cdots, \lambda_t)$ be a sequence consisting of pairwise distinct points on the projective line $\mathbb{P}_{\mathbf k}^1$ over $\mathbf{k}$, normalized as $\lambda_1=\infty$, $\lambda_2=0$ and $\lambda_3=1$. A weighted projective line $\mathbb{X}(\mathbf{p}; {\boldsymbol\lambda})$ of weight type $\mathbf{p}$ and parameter sequence ${\boldsymbol\lambda}$ is obtained from $\mathbb{P}_{\mathbf k}^1$ by attaching the weight $p_i$ to each point $\lambda_i$ for $1\leq i\leq t$.
We simply write $\bbX(\bf p; {\boldsymbol\lambda})=\bbX(\bf p)$ when $t\leq 3$. Denote by $S(\mathbf{p}; {\boldsymbol\lambda})$ (\emph{resp.} ${\rm coh}\mbox{-}\mathbb{X}(\mathbf{p}; {\boldsymbol\lambda})$) the homogeneous coordinate algebra (\emph{resp.} the category of coherent sheaves) on the weighted projective line $\mathbb{X}(\mathbf{p}; {\boldsymbol\lambda})$.
Let $H$ be a finite subgroup of $\bbL({\bf p})$. There is an $H$-action on ${\rm coh}\mbox{-}\mathbb{X}(\mathbf{p}; {\boldsymbol\lambda})$ given by degree-shifts, and the associated equivariant category will be denoted by $\big({\rm coh}\mbox{-}\mathbb{X}(\mathbf{p}; {\boldsymbol\lambda})\big)^H$.

We have the following main result of this paper, summarizing  Theorem \ref{from admissible to equivariant} and Theorem \ref{from equivariant to admissible}.

\begin{Thm-A}\label{theorem for equivariant introduction}
Let $H$ be a finite subgroup of the string group $\bbL(\bf p)$.
Then the following statements are equivalent:

\begin{itemize}
    \item [(1)] there exists an admissible homomorphism $\pi\colon \mathbb{L}(\mathbf{p})\rightarrow \mathbb{L}(\mathbf{q})$ with $\ker\pi=H$;
  \item [(2)] for any parameter sequence ${\boldsymbol\lambda}$,
    there exists a parameter sequence ${\boldsymbol\mu}$,
   s.t.  $$({\rm coh}\mbox{-}\mathbb{X}(\mathbf{p}; {\boldsymbol\lambda}))^{H}\stackrel{\sim}\longrightarrow {\rm coh}\mbox{-}\mathbb{X}(\mathbf{q}; {\boldsymbol\mu}).$$
\end{itemize}
\end{Thm-A}

Theorem A indicates that the equivariant relations between weighted projective lines induced by degree-shift actions are characterized by the admissible homomorphisms between the associated string groups. On the other hand, an admissible homomorphism $\pi\colon \bbL(\mathbf{p})\rightarrow \bbL(\mathbf{q})$ is essentially determined by its kernel in the sense of Proposition \ref{from subgroup to admissible homomorphism}, which is a finite subgroup of the string group $\bbL(\mathbf{p})$ of Cyclic type or Klein type (see Proposition \ref{ker pi}). Basing on this fact, we have the following classification result, see Theorem \ref{classification}.

\begin{Thm-B}\label{theorem on admissible in troduction} Assume $\bbL(\mathbf{p})$ and $\bbL(\mathbf{q})$ are both of domestic types or tubular types. Then all the admissible homomorphisms $\pi: \bbL(\bf p)\to \bbL(\bf q)$ are given in Table \ref{table for domestic admissible} (on page 14) and Table \ref{table for tubular admissible} (on page 15) respectively.
\end{Thm-B}

It is not sensible to state such a classification result for wild types. Instead, we produce a procedure to determine all the admissible homomorphisms $\pi: \bbL(\mathbf{p})\to \bbL(\mathbf{q})$'s for any given weight type $\mathbf{p}$ in Section \ref{subsection of classification of adm}, and provide a typical example in Example \ref{example for 46710}.

As an immediate consequence of Theorem A and Theorem B, we can obtain all the equivariant equivalences induced by degree-shift actions between weighted projective lines of domestic types, which can be expressed in the following Figure.

\begin{figure}[H]
\scriptsize
$$
\begin{array}{c}
\xymatrix{
(2,3,4)\ar[d]_{C_2}&&&&&\\
(2,3,3)\ar[d]_{C_3}&(2,2,2p_3)\ar[d]^{C_2}\ar@/^1pc/[ddr]^{C_2\times C_2}&&&(nq,n)\ar[ddd]^{C_n}&(nq_1,nq_2)\ar[ddd]^{C_n}\\
(2,2,2)\ar[d]_{C_2}\ar@/^1pc/[ddr]^{C_2\times C_2}&(2,2,p_3)\ar[rd]^{C_2}&&(np_3,np_3)\ar[dl]^{C_n}&&\\
(2,2)\ar[dr]_{C_2}&&(p_3,p_3)\ar[dl]^{C_{p_3}}&&&\\
&()&&&(q)&(q_1,q_2)
}\\
\end{array}
$$
\caption{Equivaraiant relations between domestic types}\label{fig:admissible for domestic introduction}
\end{figure}
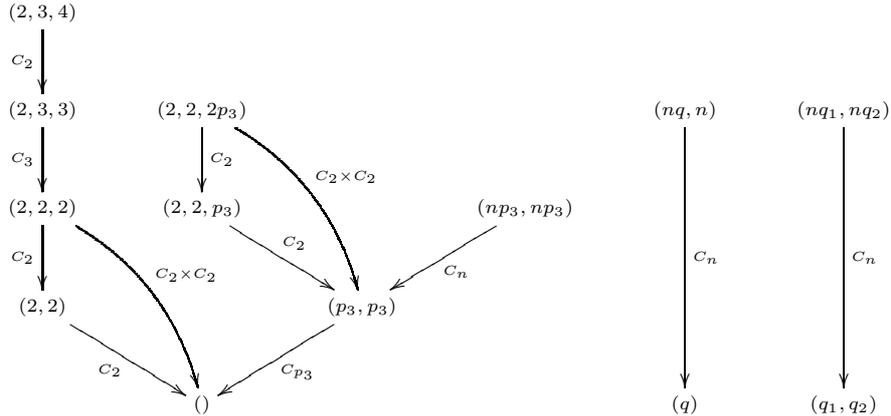
\noindent Here, a weight symbol $(a,b,c)$ stands for the (isoclass of the) weighted projective line $\mathbb{X}(a,b,c)$; an arrow $\xymatrix{\mathbb{X}({\bf p})\ar[r]^{H} &\mathbb{X}({\bf q})}$ stands for an equivalence $(\coh{\mathbb{X}({\bf p})})^{H}\stackrel{\sim}\longrightarrow\coh{\mathbb{X}({\bf q})}$; and the symbol $C_m$ (resp. $C_2\times C_2$) stands for a finite subgroup of $\bbL({\bf p})$ of Cyclic type with order $m$ (resp. of Klein type).

The main connected component of Figure \ref{fig:admissible for domestic introduction} has closed relation to \cite[Figure 1]{Len16}, more precisely, its dual graph coincides with the full subgraph of \cite[Figure 1]{Len16} consisting of those arrows marked by \textbf{abelian} groups (an arrow connecting $(2,2,2n)$ and $(n,n)$ is missing there). For an explanation of this `duality' we refer to \cite[Theorem 4.6]{CCR}, which states that for a finite \textbf{abelian} group action on a linear category, the dual action (given by the character group acting on the equivariant category) recovers the original category, compare \cite[Theorem 4.2]{El2014} and \cite[Theorem 4.4]{DGNO}.

Now we consider the equivariant relations between weighted projective lines of tubular types.
Recall that a weight sequence $\mathbf{p}$ is of tubular type if and only if $\mathbf{p}=(2,2,2,2)$, $(3,3,3)$, $(4,4,2)$ or $(6,3,2)$ up to permutation. By our normalization assumption, we can write a weighted projective line of type (2,2,2,2) as $\mathbb{X}(2,2,2,2;\lambda)$ with $\lambda\in\mathbf{k}\backslash\{0,1\}$. For the classifications of the equivariant relations between tubular types, the parameter $\lambda$ plays a key role.
For convenience we denote by $\Gamma(\lambda)$ the multiset $\{\lambda, \frac{1}{\lambda},1-\lambda,\frac{1}{1-\lambda},\frac{\lambda}{\lambda-1}, \frac{\lambda-1}{\lambda}\}$ and let $f(x)=(\frac{x+1}{x-1})^2$ for $x\neq 1$. Then all the equivariant equivalences induced by degree-shift actions between weighted projective lines of tubular types are listed below:

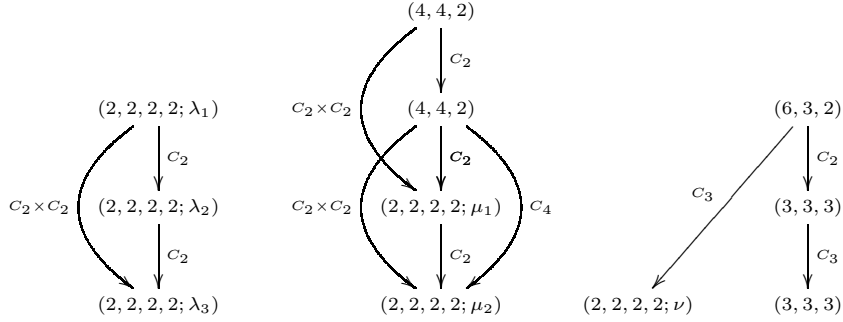
\begin{figure}[H]
\scriptsize
$$\xymatrix{
&&(4,4,2)\ar[d]^{C_2}\ar@/_2.5pc/[dd]_{C_2\times C_2}&&&&\\
(2,2,2,2;\lambda_1)\ar[d]^{C_2}\ar@/_2.5pc/[dd]_{C_2\times C_2}&&(4,4,2)\ar[d]^{C_2}\ar@/_2.5pc/[dd]_{C_2\times C_2}\ar[d]^{C_2}\ar@/^2.5pc/[dd]^{C_4}&&(6,3,2)\ar[ddl]_{C_3}\ar[d]^{C_2}&&\\
(2,2,2,2;\lambda_2)\ar[d]^{C_2}&&(2,2,2,2;\mu_1)\ar[d]^{C_2}&&(3,3,3)\ar[d]^{C_3}&\\
(2,2,2,2;\lambda_3)&&(2,2,2,2;\mu_2)&(2,2,2,2;\nu)&(3,3,3)
}$$
\caption{Equivaraiant relations between tubular types}
\end{figure}

\noindent where the parameters satisfy the following relations: $\lambda_1, \lambda_2, \lambda_3$ are arbitrary with $\Gamma(\lambda_{1})=\Gamma(\lambda_{3})$, while $\Gamma(\mu_1)=\Gamma(\mu_2)=\Gamma(-1)$ and $\Gamma(\nu)=\Gamma(\frac{1+ \sqrt{-3}}{2})$; moreover, for each arrow $(2,2,2,2;\lambda)\xrightarrow{C_2}(2,2,2,2;\mu)$, we require $\Gamma(\mu)=\Gamma(f(\sqrt{\lambda'}))$ for some $\lambda'\in\Gamma(\lambda)$.

According to Proposition \ref{ker pi} and Proposition \ref{from subgroup to admissible homomorphism}, the statement (1) or (2) in Theorem A holds if and only if $H=0$ or $H$ is of Cyclic type or Klein type. Then a natural question is:
what's the type of the equivariant category $({\rm coh}\mbox{-}\mathbb{X}(\mathbf{p}; {\boldsymbol\lambda}))^{H}$ for general subgroup $H$'s? For the answer we refer to the forthcoming paper, where the notion of weighted projective curves in the sense of Lenzing \cite{Len16} will be involved.

The paper is organized as follows.
In Section 2, we recall some basic facts on weighted projective lines and review the definition of a group action on a category.
Section 3 {focuses} on the study of the properties of admissible homomorphisms. We prove that admissible homomorphisms preserve domestic, tubular and wild types respectively due to an explicit description of their kernels. Moreover, we show that an admissible homomorphism is determined by its kernel in certain sense, which helps us to classify all the admissible homomorphisms for domestic and tubular types. The main result of this paper is stated in Section 4. We show that the equivariant relations induced by degree-shift actions between weighted projective lines are characterized by admissible homomorphisms between the associated string groups. As an application, we classify all the equivariant equivalences between the categories of coherent sheaves for domestic and tubular types in Section 5.

\section{Weighted projective lines and admissible homomorphisms}

In this section, we recall from \cite{GL87,GL90, CC17} some basic facts on weighted projective lines and admissible homomorphisms between string groups.

\subsection{String group}

Let $t\geq 0$ be an integer. Let ${\bf p}=(p_1, p_2, \cdots, p_t)$ be a sequence
of integers with each $p_i\geq 2$.
The \emph{string group} $\bbL({\bf p})$ of type ${\bf p}$ is an abelian group (written additively) on generators $\vec{x}_1, \vec{x}_2, \cdots, \vec{x}_t$, subject to the relations $p_1\vec{x}_1=p_2\vec{x}_2=\cdots = p_t\vec{x}_t$, where this common element is denoted by $\vec{c}$ and called the \emph{canonical element} of $\bbL(\mathbf{p})$.
The string group $\bbL(\mathbf{p})$ has rank one, where $\vec{c}$ has infinite order. There is an isomorphism of abelian groups
\begin{align}\label{isomorphism of L} \bbL(\mathbf{p})/\mathbb{Z}\vec{c}\stackrel{\sim}\longrightarrow \prod_{i=1}^t \mathbb{Z}/p_i\mathbb{Z},\end{align}
sending $\vec{x}_i+\mathbb{Z}\vec{c}$ to the vector $(0, \cdots,0,  \bar{1}, 0, \cdots, 0)$ with $\bar{1}$ on the $i$-th component. Using (\ref{isomorphism of L}), each element $\vec{x}$ in $\bbL(\mathbf{p})$ can be uniquely written in its \emph{normal form}
\begin{align}\label{equ:nor}
\vec{x}=\sum_{i=1}^t l_i\vec{x}_i+l\vec{c},
\end{align}
where $0\leq l_i\leq p_i-1$ for $1\leq i\leq t$ and $l\in \mathbb{Z}$. With this normal form, we define two functions $$\mu(\vec{x}):=|\{\,i\,|\,l_i\neq 0, \,1\leq i\leq t\}| \text{\quad
and\quad}  {\rm mult}(\vec{x}):={\rm max}\{l+1, 0\},$$
here and forward we use $|\cS|$ to denote the cardinality of a finite set $\cS$.

Denote by $p={\rm l.c.m.}(p_1,p_2,\cdots,p_t)$. There is a surjective group homomorphism $\delta\colon \bbL(\mathbf{p})\rightarrow \mathbb{Z}$ given by $\delta(\vec{x}_i)=\frac{p}{p_i}$ for $1\leq i\leq t$. The torsion group $t\bbL(\mathbf{p})$ of the string group $\bbL(\mathbf{p})$ coincides with $\ker\delta$, which has cardinality $\frac{\prod_{i=1}^tp_i}{p}$.
For each $1\leq i\leq t$, there is a surjective group homomorphism $\pi_i\colon \bbL(\mathbf{p})\rightarrow \mathbb{Z}/p_i\mathbb{Z}$ with $\pi_i(\vec{x}_j)=\delta_{i,j}\bar{1}$. Here, $\delta_{i,j}$ is the Kronecker symbol. By \cite[Definition 6.5]{CCZ}, an infinite subgroup $H\subseteq \bbL(\mathbf{p})$ is called \emph{effective} if $\pi_i(H)=\mathbb{Z}/p_i\mathbb{Z}$ for each $i$.

Recall that the \emph{dualizing element} $\vec{\omega}$ in $\bbL(\mathbf{p})$ is defined as $\vec{\omega}=(t-2)\vec{c}-\sum_{i=1}^t \vec{x}_i$. Hence we have $\delta(\vec{\omega})=p((t-2)-\sum_{i=1}^t\frac{1}{p_i})$.
The string group $\bbL(\mathbf{p})$ is called of \emph{domestic, tubular or wild type} provided that $\delta(\vec{\omega})<0$, $\delta(\vec{\omega})=0$ or $\delta(\vec{\omega})>0$ respectively. More precisely, we have the following trichotomy for the string groups $\bbL(\mathbf{p})$ according to the type $\mathbf{p}$ (up to permutation):
\begin{itemize}
  \item [(i)] domestic type: $(), (p), (p_{1},p_{2})$, $(2,2,p_3)$, $(2,3,3)$, $(2,3,4)$ and $(2,3,5)$;
  \item [(ii)] tubular type: $(2,2,2,2)$, $(3,3,3)$, $(4,4,2)$  and $(6,3,2)$;
  \item [(iii)] wild type: all the other cases.
\end{itemize}

\subsection{Weighted projective line}

Let $\mathbf{k}$ be an algebraically closed field. Let ${\boldsymbol\lambda}=(\lambda_1, \lambda_2, \cdots, \lambda_t)$ be a sequence of
pairwise distinct points on the projective line $\mathbb{P}_{\mathbf k}^1$. Such a sequence can be normalized such that $\lambda_1=\infty$, $\lambda_2=0$ and $\lambda_3=1$.
A \emph{weighted projective line} $\mathbb{X}(\mathbf{p}; {\boldsymbol\lambda})$ of weight type $\mathbf{p}$ and parameter sequence ${\boldsymbol\lambda}$ is obtained from $\mathbb{P}_{\mathbf k}^1$ by attaching the weight $p_i$ to each point $\lambda_i$ for $1\leq i\leq t$. We will always assume ${\boldsymbol\lambda}$ is normalized in this paper unless stated otherwise,
and simply write $\bbX(\bf p; {\boldsymbol\lambda})=\bbX(\bf p)$ for $t\leq 3$.

The \emph{homogeneous coordinate algebra} $S(\mathbf{p}; {\boldsymbol\lambda})$ of the weighted projective line $\mathbb{X}(\mathbf{p}; {\boldsymbol\lambda})$ is given by $\mathbf{k}[X_1, X_2, \cdots, X_t]/I$, where the ideal $I$ is generated by $X_i^{p_i}-(X_2^{p_2}-\lambda_iX_1^{p_1})$ for $3\leq i\leq t$. We write $x_i=X_i+I$ in $S(\mathbf{p}; {\boldsymbol\lambda})$.
The algebra $S(\mathbf{p}; {\boldsymbol\lambda})$ is $\bbL(\mathbf{p})$-graded by means of $\deg x_i=\vec{x}_i$. Then we have $S(\mathbf{p}; {\boldsymbol\lambda})=\bigoplus_{\vec{x}\in \bbL(\mathbf{p})} S(\mathbf{p}; {\boldsymbol\lambda})_{\vec{x}}$, where $S(\mathbf{p}; {\boldsymbol\lambda})_{\vec{x}}$ denotes the homogeneous component of degree $\vec{x}$. Write $\vec{x}$ in its normal form (\ref{equ:nor}), then $S(\mathbf{p}; {\boldsymbol\lambda})_{\vec{x}}\neq 0$ if and only if $l\geq 0$. Moreover, by \cite[Proposition 1.3]{GL87}, $\{x_1^{ap_1}x_2^{bp_2}x_1^{l_1}x_2^{l_2}\cdots x_t^{l_t}\; |\; a+b=l, a, b\geq 0\}$ form a ${\mathbf k}$-basis of $S(\mathbf{p}; {\boldsymbol\lambda})_{\vec{x}}$. We deduce that
$${\rm dim}_{\mathbf k}\; S(\mathbf{p}; {\boldsymbol\lambda})_{\vec{x}}={\rm mult}(\vec{x}), \mbox{ for all } \vec{x}\in \bbL(\mathbf{p}).$$
For an infinite subgroup $H\subseteq \bbL(\mathbf{p})$, we have a \emph{restriction subalgebra} $S(\mathbf{p}; {\boldsymbol\lambda})_H=\bigoplus_{\vec{x}\in H}S(\mathbf{p}; {\boldsymbol\lambda})_{\vec{x}}$ of $S(\mathbf{p}; {\boldsymbol\lambda})$. According to \cite[Lemma 6.2]{CCZ}, $S(\mathbf{p}; {\boldsymbol\lambda})_H$ is a finitely generated $H$-graded $\mathbf{k}$-algebra.

We recall the definition of the category ${\rm coh}\mbox{-}\mathbb{X}(\mathbf{p}; {\boldsymbol\lambda})$ of coherent sheaves over $\mathbb{X}(\mathbf{p}; {\boldsymbol\lambda})$ by a more convenient description via graded $S(\mathbf{p}; {\boldsymbol\lambda})$-modules. We denote by ${\rm mod}^{\bbL(\mathbf{p})}\mbox{-}S(\mathbf{p}; {\boldsymbol\lambda})$ the abelian category of finitely generated $\bbL(\mathbf{p})$-graded $S(\mathbf{p}; {\boldsymbol\lambda})$-modules, and by ${\rm mod}_0^{\bbL(\mathbf{p})}\mbox{-}S(\mathbf{p}; {\boldsymbol\lambda})$ its Serre subcategory formed by finite dimensional modules. We denote by ${\rm qmod}^{\bbL(\mathbf{p})}\mbox{-}S(\mathbf{p}; {\boldsymbol\lambda}):={\rm mod}^{\bbL(\mathbf{p})}\mbox{-}S(\mathbf{p}; {\boldsymbol\lambda})/{{\rm mod}_0^{\bbL(\mathbf{p})}\mbox{-}S(\mathbf{p}; {\boldsymbol\lambda})}$ the quotient abelian category. By \cite[Theorem 1.8]{GL87} the sheafification functor yields an equivalence
$$
{\rm qmod}^{\bbL(\mathbf{p})}\mbox{-}S(\mathbf{p}; {\boldsymbol\lambda})\stackrel{\sim}\longrightarrow {\rm coh}\mbox{-}\mathbb{X}(\mathbf{p}; {\boldsymbol\lambda}).
$$
From now on we will identify these two categories.

\subsection{Admissible homomorphism}
Throughout this paper we always assume that $\bbL(\mathbf{p})$ and $\bbL(\mathbf{q})$ are two string groups with ${\bf p}=(p_1, p_2, \cdots, p_t)$ and $\mathbf{q}=(q_1,q_2, \cdots, q_s)$. Denote by $p={\rm l.c.m.}(p_1, p_2, \cdots, p_t)$ and $q={\rm l.c.m.}(q_1, q_2, \cdots, q_s)$, and write
$$\bbL(\mathbf{p})=\bbZ\{\vec{x}_1, \vec{x}_2, \cdots, \vec{x}_t\}/(p_1\vec{x}_1=p_2\vec{x}_2=\cdots = p_t\vec{x}_t:=\vec{c})$$
and $$\bbL(\mathbf{q})=\bbZ\{\vec{z}_1, \vec{z}_2, \cdots, \vec{z}_s\}/(q_1\vec{z}_1=q_2\vec{z}_2=\cdots = q_s\vec{z}_s:=\vec{d}).$$

\begin{defn}\label{defn:AH}\cite{CC17}
A group homomorphism $\pi\colon \bbL(\mathbf{p})\rightarrow \bbL(\mathbf{q})$ is called \emph{admissible} if the following conditions are satisfied:
\begin{enumerate}
\item[(1)]  the subgroup $\im\pi\subseteq \bbL(\mathbf{q})$ is effective;
\item[(2)]  for each $\vec{z}\in \im\pi$, $\sum\limits_{\vec{x}\in \pi^{-1}(\vec{z})} {\rm mult}(\vec{x})={\rm mult}(\vec{z})$.
\end{enumerate}
\end{defn}

We point out that if $\pi$ is admissible, then $\ker\pi$ is a subgroup of the torsion group $t\bbL(\mathbf{p})$ of $\bbL(\mathbf{p})$, in particular, $|\ker\pi|$ is finite. Indeed, since $\bbL(\mathbf{p})$ has rank one and $\im\pi$ is infinite by effectiveness of $\pi$, we obtain that $\im\pi$ also has rank one and therefore $\ker\pi\subseteq t\bbL(\mathbf{p})$.

Let $\mathbb{X}(\mathbf{p}; {\boldsymbol\lambda})$ and $\mathbb{X}(\mathbf{q}; {\boldsymbol\mu})$ be weighted projective lines and $\pi\colon \bbL(\mathbf{p})\rightarrow \bbL(\mathbf{q})$ be a group homomorphism with $\im\pi$ infinite. An algebra homomorphism $\phi\colon S(\mathbf{p}; {\boldsymbol\lambda})\rightarrow S(\mathbf{q}; {\boldsymbol\mu})$ is called \emph{compatible} with $\pi$ if $\phi(S(\mathbf{p}; {\boldsymbol\lambda})_{\vec{x}})\subseteq S(\mathbf{q}; {\boldsymbol\mu})_{\pi(\vec{x})}$ for each $\vec{x}\in \bbL(\mathbf{p})$. Observe that $\phi$ induces a homomorphism between $\im\pi$-graded algebras
\begin{align}\label{equ:phi}
\bar{\phi}\colon \pi_*S(\mathbf{p}; {\boldsymbol\lambda})\longrightarrow S(\mathbf{q}; {\boldsymbol\mu})_{\im\pi};\quad x\mapsto \phi(x).
\end{align}
Here, $\pi_*S(\mathbf{p}; {\boldsymbol\lambda})=S(\mathbf{p}; {\boldsymbol\lambda})$ as an ungraded algebra, and the homogeneous component $(\pi_*S(\mathbf{p}; {\boldsymbol\lambda}))_{\vec{z}}=\bigoplus_{\vec{x}\in \pi^{-1}(\vec{z})} S(\mathbf{p}; {\boldsymbol\lambda})_{\vec{x}}$ for each $\vec{z}\in \im\pi$; $S(\mathbf{q}; {\boldsymbol\mu})_{\im\pi}$ is the restriction subalgebra of $S(\mathbf{q}; {\boldsymbol\mu})$ with respect to the subgroup $\im\pi\subseteq \bbL(\mathbf{q})$.

\begin{lem}\label{lem:phi}\cite{CC17}
 Assume that $\bar{\phi}$ in (\ref{equ:phi}) is surjective and $\pi\colon \bbL(\mathbf{p})\rightarrow \bbL(\mathbf{q})$ satisfies the condition (2) in Definition \ref{defn:AH}, then $\bar{\phi}$ is an isomorphism.
\end{lem}

\subsection{Equivariant category}

Now we are going to introduce how to relate two weighted projective lines via equivariantization. Firstly we recall from \cite{DGNO,De,CCZ} the equivariantization briefly.

Let $G$ be a group with unit $e$. Temporarily, we write $G$ mutliplicatively. A \emph{strict} $G$-action on a category $\mathcal{A}$ is a group homomorphism from $G$ to the automorphism group of $\mathcal{A}$, which assigns for each $g\in G$ an automorphism $F_g$ of $\mathcal{A}$. Hence, we have $F_e={\rm Id}_\mathcal{A}$ and $F_g F_h=F_{gh}$ for all $g,h\in G$.

A \emph{$G$-equivariant object} in $\mathcal{A}$ is a pair $(X, \alpha)$, where $X$ is an object in $\mathcal{A}$ and $\alpha$ assigns to each $g\in G$ an isomorphism $\alpha_g\colon X\rightarrow F_g(X)$ subject to the relations $\alpha_{gh}=F_g(\alpha_h)\circ \alpha_g$. A morphism $f\colon (X, \alpha)\rightarrow (Y, \beta)$ between equivariant objects is a morphism $f\colon X\rightarrow Y$ in $\mathcal{A}$ satisfying $\beta_g\circ f=F_g(f)\circ \alpha_g$. This gives rise to the \emph{equivariant category} $\mathcal{A}^G$ of equivariant objects, and the \emph{forgetful functor} $U: \mathcal{A}^G\to \mathcal{A}$ defined by $U(X, \alpha) =X$. The process forming the equivariant category $\mathcal{A}^G$ is known as the equivariantization with respect to the group action; see \cite{DGNO}.
Observe that if $\mathcal{A}$ is abelian, then so is $\mathcal{A}^G$. Indeed, a sequence of equivariant objects is exact in $\mathcal{A}^G$ if and only if so is the sequence of underlying objects in $\mathcal{A}$.

In what follows, we assume that the group G is finite and that $\mathcal{A}$ is an additive category. In this case, the forgetful functor $U$ admits a left adjoint $F: \mathcal{A}\to \mathcal{A}^G$, which is known as the \emph{induction functor}; see \cite[Lemma 4.6]{DGNO}. The functor $F$ is defined as follows: for an object $X$, set $F(X) = (\bigoplus_{h\in G} F_h(X), \varepsilon)$, where, for each $g\in G$, the isomorphism $\varepsilon_{g}: \bigoplus_{h\in G} F_h(X)\to  F_{g}(\bigoplus_{h\in G} F_h(X))$ is diagonally induced by the isomorphism $(\varepsilon_{g,g^{-1}h})^{-1}_X: F_h(X)\to F_{g}(F_{g^{-1}h}(X))$.

Now we consider a group action on the category ${\rm coh}\mbox{-}\mathbb{X}(\mathbf{p}; {\boldsymbol\lambda})$. For each subgroup $G\subseteq \bbL(\mathbf{p})$, we have a strict $G$-action on ${\rm mod}^{\bbL(\mathbf{p})}\mbox{-}S(\mathbf{p}; {\boldsymbol\lambda})$ by setting $F_{\vec{x}}=(\vec{x})$ for each $\vec{x}\in G$. Here, $(\vec{x})$ is the degree-shift action by the element $\vec{x}$. This $G$-action induces a strict $G$-action on ${\rm coh}\mbox{-}\mathbb{X}(\mathbf{p}; {\boldsymbol\lambda})$, yielding an equivariant category $({\rm coh}\mbox{-}\mathbb{X}(\mathbf{p}; {\boldsymbol\lambda}))^G$.

\begin{prop}\label{Prop:2.3}\cite{CC17}
Let $\pi\colon \bbL(\mathbf{p})\rightarrow \bbL(\mathbf{q})$ be an admissible homomorphism. Assume that the algebra homomorphism $\phi\colon S(\mathbf{p}; {\boldsymbol\lambda})\rightarrow S(\mathbf{q}; {\boldsymbol\mu})$ induces a surjective homomorphism $\bar{\phi}$ in (\ref{equ:phi}). Then $\bar{\phi}\colon \pi_*S(\mathbf{p}; {\boldsymbol\lambda})\rightarrow S(\mathbf{q}; {\boldsymbol\mu})_{\im\pi}$ is an isomorphism of $\im\pi$-graded algebras, which induces an equivalence of categories
\begin{align*}
({\rm coh}\mbox{-}\mathbb{X}(\mathbf{p}; {\boldsymbol\lambda}))^{\ker\pi}\stackrel{\sim}\longrightarrow {\rm coh}\mbox{-}\mathbb{X}(\mathbf{q}; {\boldsymbol\mu}).
\end{align*}
\end{prop}

\section{Admissible homomorphism between string groups}

In this section, we study the properties of admissible homomorphisms between string groups.
We always fix a group homomorphism
\begin{align}\label{pi}\pi\colon \bbL(\mathbf{p})\rightarrow \bbL(\mathbf{q});\quad \vec{x}_i\mapsto \sum_{j=1}^s a_{ij}\vec{z}_j+a_{i,s+1}\vec{d}\;\; \;\;(1\leq i\leq t);
\end{align}
where $0\leq a_{ij}\leq q_j-1$ and $a_{i,s+1}\in\bbZ$ for each $1\leq i\leq t, \; 1\leq j\leq s$.

\begin{prop} \label{coprime}
Keep the notation as in (\ref{pi}). Then $\im\pi$ is effective if and only if ${\rm g.c.d.}(a_{1j}, a_{2j}, \cdots, a_{tj}, q_j)=1$ for each $1\leq j\leq s$.
\end{prop}

\begin{proof} By definition, $\im\pi$ is effective if and only if $\pi_j(\im\pi )=\mathbb{Z}/q_j\mathbb{Z}$ for any $1\leq j\leq s$, if and only if there exist integers $k_{ij}$ for $1\leq i\leq t$ and $1\leq j\leq s$, such that $\sum_{i=1}^t k_{ij}a_{ij}=1\, ({\rm mod}\, q_j)$, or equivalently, $\sum_{i=1}^t k_{ij}a_{ij}+k_jq_j=1$ for some integer $k_j$, that is, ${\rm g.c.d.}(a_{1j}, a_{2j}, \cdots, a_{tj}, q_j)=1$.
\end{proof}

From now on, we always assume that $\pi$ in (\ref{pi}) is admissible.

\subsection{The kernel of $\pi$}

In this subsection, we first study the kernel of $\pi$.
Since $\ker\pi$ is a finite group, we assume
$\ker\pi=\{\vec{y}_1=0, \vec{y}_2, \cdots, \vec{y}_n\}$ with normal forms
\begin{align}\label{yi}\vec{y}_i=\sum_{j=1}^t b_{ij}\vec{x}_j+b_{i}\vec{c},\quad 2\leq i\leq n.
\end{align}
Recall that $\mu(\vec{y}_i)=|\{\,j\,|\,b_{ij}\neq 0, \,1\leq j\leq t\}|$. Obviously, $\delta(\vec{y}_i)=0$ implies $b_{i}<0$ and $\mu(\vec{y}_i)\geq 2$ for $2\leq i\leq n$. Moreover, we have

\begin{lem}\label{pi c}
Keep the notation as in (\ref{yi}). Then for $2\leq i\leq n$, we have
\begin{itemize}
  \item [(1)] $b_i=-1$;
  \item [(2)]  $\mu(\vec{y}_i)=2$.
\end{itemize}
Consequently, $\pi(\vec{c})=n\vec{d}$.
\end{lem}

\begin{proof} (1) Assume $\pi(\vec{c})=\sum_{i=1}^s l_{i}\vec{z}_i+l\vec{d}$ is in normal form. Then $\pi$ is admissible implies $l\geq 1$. Recall that $q={\rm l.c.m.} (q_1, q_2,\cdots, q_s)$. Then $\pi(q\vec{c})=q\cdot\pi(\vec{c})=b\vec{d}$ for some integer $b\geq 0$. Now for $m\gg0$, we have ${\rm mult}(mq\vec{c}+\vec{y}_1)=mq+1$ and ${\rm mult}(mq\vec{c}+\vec{y}_i)=mq+b_i+1$ for $2\leq i\leq n$. Hence
$$\sum_{\vec{x}\in \pi^{-1}(mb\vec{d})} {\rm mult}(\vec{x})=(mq+1)+\sum_{i=2}^{n}(mq+b_{i}+1)=mnq+1+\sum_{i=2}^{n}(b_{i}+1).$$
On the other hand, we have $\pi(mq\vec{c})=mb\vec{d}$ and ${\rm mult}(\pi(mq\vec{c}))=mb+1$.
Then $\pi$ is admissible implies $mb+1=mnq+1+\sum_{i=2}^{n}(b_{i}+1)$ for $m\gg0$. Therefore, $b=nq$ and $\sum_{i=2}^{n}(b_{i}+1)=0$, which forces $b_{i}=-1$ for $2\leq i\leq n$ since each $b_i<0$.

Consequently, we have ${\rm mult}(\vec{c}+\vec{y}_i)=1$ for $2\leq i\leq n$, hence
$$l+1={\rm mult}(\pi(\vec{c}))={\rm mult}(\vec{c})+\sum_{i=2}^{n}{\rm mult}(\vec{c}+\vec{y}_i)=n+1.$$ Thus $l=n$. Then $\pi(q\vec{c})=\sum_{i=1}^s ql_{i}\vec{z}_i+qn\vec{d}=b\vec{d}=nq\vec{d}$ implies that $l_{i}=0$ for $1\leq i\leq s$. Hence $\pi(\vec{c})=n\vec{d}$.

(2) Assume $\pi(\vec{\omega})=\sum_{i=1}^s a_{i}\vec{z}_i+a\vec{d}$ is in normal form.
By (1) we obtain that
${\rm mult}(\pi(\vec{\omega}+m\vec{c}))={\rm mult}\big(\pi(\vec{\omega})+mn\vec{d} \big)=mn+a+1$ for $m\gg 0$. On the other hand, ${\rm mult}(\vec{\omega}+m\vec{c})=m-1$ and ${\rm mult}(\vec{\omega}+m\vec{c}+\vec{y}_i)=m+\mu(\vec{y}_i)-2$ for $2\leq i\leq n$.
Then $\pi$ is admissible implies $$mn+a+1=m-1+\sum_{i=2}^{n}(m+\mu(\vec{y}_i)-2),\quad m\gg 0.$$
Thus $\sum_{i=2}^{n}(\mu(\vec{y}_i)-2)=a+2$.

Observe that ${\rm mult}(\pi(\vec{\omega}))={\rm max}\{a+1,0\}$,
${\rm mult}(\vec{\omega})=0$ and ${\rm mult}(\vec{\omega}+\vec{y}_i)=\mu(\vec{y}_i)-2$ for $2\leq i\leq n$. Then $\pi$ is admissible yields
$${\rm max}\{a+1,0\}={\rm mult}(\pi(\vec{\omega}))=\sum_{\vec{x}\in \pi^{-1}(\pi(\vec{\omega}))} {\rm mult}(\vec{x})=0+\sum_{i=2}^{n}(\mu(\vec{y}_i)-2)=a+2,$$ which implies $a=-2$. It follows that $\mu(\vec{y}_i)$=2 for $2\leq i\leq n$.
\end{proof}

Recall that ${\bf p}=(p_1, p_2, \cdots, p_t)$.
Let $\sigma$ be a permutation on $\{1,2,\cdots,t\}$. Let $\pi_{\sigma}$ be the automorphism on the group $\bbL({\bf p})$ defined by $\pi_{\sigma}(\vec{x}_i)=\vec{x}_{\sigma(i)}$ for $1\leq i\leq t$. It is easy to check that $\pi_{\sigma}$ is admissible with $\ker\pi_{\sigma}=0$. On the other hand, we have

\begin{lem} \label{cor:3.5}
If $\ker\pi=0$, then $\pi=\pi_{\sigma}$ for some permutation $\sigma$ on $\{1,2,\cdots,t\}$. That is, up to permutation we have $\bf p=\bf q$ and $\pi={\rm id}$.
\end{lem}

\begin{proof}
By Lemma \ref{pi c}, we get $\pi(\vec{c})=\vec{d}$. Then $\pi$ is admissible implies $\pi(\vec{x}_i)=a_{i}\vec{z}_{\sigma(i)}$ for $1\leq i\leq t$, where $1\leq \sigma(i)\leq s$ and $1\leq a_{i}<q_{\sigma(i)}$. This defines a map $\sigma: \{1,2,\cdots, t\}\to\{1,2,\cdots, s\}$. Now $\ker\pi=0$ implies $\sigma$ is injective, while $\im\pi$ is effective implies $\sigma$ is surjective.
Hence $\sigma$ is bijective and hence $s=t$.

Up to permutation, we can assume $\pi(\vec{x}_i)=a_{i}\vec{z}_i$ with $1\leq a_i<q_i$ for $1\leq i\leq t$. Then $\pi(\vec{c})=\vec{d}$ implies $p_ia_i=q_i$ for each $i$. According to Proposition \ref{coprime}, we have $a_i=(a_i, q_i)=1$ and then $p_i=q_i$ for each $i$. It follows that $\bf p=\bf q$ and $\pi={\rm id}$.
\end{proof}

In the following we assume $\pi$ is admissible with $\ker\pi\neq 0$ unless stated otherwise. We denote by $C_n$ the cyclic group of order $n$.
The following result plays a key role throughout this paper.

\begin{prop} \label{ker pi} Up to permutation, $\ker\pi$ has one of the following types:
 \begin{itemize}
   \item[(1)] Cyclic type: $\langle \frac{p_1}{n}\vec{x}_1-\frac{p_2}{n}\vec{x}_2\rangle$ with $n\, |\, {\rm g.c.d.}(p_1, p_2)$;
   \item[(2)] Klein type: $\langle \frac{p_1}{2}\vec{x}_1-\frac{p_2}{2}\vec{x}_2, \frac{p_1}{2}\vec{x}_1-\frac{p_3}{2}\vec{x}_3\rangle$ with $p_1, p_2, p_3$ even.
 \end{itemize}
Moreover, for Cyclic type we have $\ker\pi\cong C_{n}$ and $\pi(\frac{p_1}{n}\vec{x}_1)=\pi(\frac{p_2}{n}\vec{x}_2)=\vec{d}$; while for Klein type we have $\ker\pi\cong C_2\times C_2$ and $\pi(\frac{p_1}{2}\vec{x}_1)=\pi(\frac{p_2}{2}\vec{x}_2)=\pi(\frac{p_3}{2}\vec{x}_3)
=2\vec{d}$.
\end{prop}

\begin{proof} By Lemma \ref{pi c}, up to permutation we can write $\vec{y}_2=b_{21}\vec{x}_1+b_{22}\vec{x}_2-\vec{c}$ in normal form.
For any $0\neq \vec{y}_i\in\ker\pi$, say, $\vec{y}_i=b_{ij}\vec{x}_j+b_{ik}\vec{x}_k-\vec{c}$, we claim that $\{1,2\}\cap \{j,k\}\neq \emptyset$. Otherwise, we have $\vec{y}_2+\vec{y}_i\in \ker\pi$ with $\mu(\vec{y}_2+\vec{y}_i)=4$, contradicting to Lemma \ref{pi c}. Therefore, we have the following two cases.
\begin{itemize}
  \item [(1)] Each nonzero element in $\ker\pi$ has the form $\vec{y}_i=b_{i1}\vec{x}_1+b_{i2}\vec{x}_2-\vec{c}$ for $2\leq i\leq n$. In this case, all the $b_{i1}$'s are pairwise distinct, hence we can assume $b_{21}<b_{31}<\cdots<b_{n1}$. We claim that $b_{21}\,|\,b_{i1}$ for $2\leq i\leq n$. For contradiction we assume $b_{i1}=b_{21}q_{i1}+r_{i1}$ with $0<r_{i1}<b_{21}$ for some $i$. Then $0\neq \vec{y}_i-q_{i1}\vec{y}_2\in\ker\pi$ with smaller coefficient of $\vec{x}_1$ in its normal form than $\vec{y}_2$, a contradiction. Hence $b_{21}\,|\,b_{i1}$ and it follows that $\vec{y}_i=\frac{b_{i1}}{b_{21}}\vec{y}_2$ for any $2\leq i\leq n$. Therefore, $\ker\pi$ is generated by $\vec{y}_2$ with order $n$. Then it follows that $\ker\pi=\langle \frac{p_1}{n}\vec{x}_1-\frac{p_2}{n}\vec{x}_2\rangle\cong C_{n}$.
  \item [(2)] There exists some $\vec{y}_i=b_{ij}\vec{x}_j+b_{ik}\vec{x}_k-\vec{c}\in\ker\pi$ with $\{j,k\}\neq\{1,2\}$. Up to permutation we can assume $(j,k)=(1,3)$ and then $\vec{y}_i=b_{i1}\vec{x}_1+b_{i3}\vec{x}_3-\vec{c}$. Hence, $0\neq \vec{y}_2\pm\vec{y}_i\in \ker\pi$. Then $\mu(\vec{y}_2\pm\vec{y}_i)=2$ implies $b_{21}+b_{i1}=p_1$ and  $b_{21}-b_{i1}=0$. It follows that $b_{21}=b_{i1}=\frac{p_1}{2}$. Now $2\vec{y}_2, 2\vec{y}_i\in\ker\pi$ implies that $b_{22}=\frac{p_2}{2}$ and $b_{i3}=\frac{p_3}{2}$ respectively.
      It follows that
      $\ker\pi=\langle \frac{p_1}{2}\vec{x}_1-\frac{p_2}{2}\vec{x}_2, \frac{p_1}{2}\vec{x}_1-\frac{p_3}{2}\vec{x}_3\rangle\cong C_2\times C_2.$
\end{itemize}

For case (1), assume $\pi(\frac{p_1}{n}\vec{x}_1)=\sum_{i=1}^sl_{i}\vec{z}_i+l\vec{d}$ in normal form.
Observe that
${\rm mult}(\frac{p_1}{n}\vec{x}_1)=1={\rm mult}(\frac{p_1}{n}\vec{x}_1+\vec{y}_n)$ and ${\rm mult}(\frac{p_1}{n}\vec{x}_1+\vec{y}_i)=0$ for $2\leq i\leq n-1$.
Then $\pi$ is admissible implies that $$2=\sum_{k=1}^n{\rm mult}(\frac{p_1}{n}\vec{x}_1+\vec{y}_i)
={\rm mult}(\pi(\frac{p_1}{n}\vec{x}_1))={\rm max}\{l+1, 0\}.$$ Thus $l=1$. Moreover, by Lemma \ref{pi c} we have $n\vec{d}=\pi(\vec{c})=n\cdot\pi(\frac{p_1}{n}\vec{x}_1)
=n(\sum_{i=1}^sl_{i}\vec{z}_i+\vec{d})$, which implies $l_i=0$ for $1\leq i\leq s$. Therefore, $\pi(\frac{p_1}{n}\vec{x}_1)=\pi(\frac{p_2}{n}\vec{x}_2)=\vec{d}$.

For case (2), we write $\pi(\frac{p_1}{2}\vec{x}_1)=\sum_{i=1}^sl_{i}\vec{z}_i+l\vec{d}$ in normal form. Then we have $3=\sum_{\vec{x}\in \pi^{-1}(\pi(\frac{p_1}{2}\vec{x}_1))} {\rm mult}(\vec{x})={\rm mult}(\pi(\frac{p_1}{2}\vec{x}_1))={\rm max}\{l+1, 0\}$. Thus $l=2$. Moreover, by Lemma \ref{pi c} we have $4\vec{d}=\pi(\vec{c})=2\cdot\pi(\frac{p_1}{2}\vec{x}_1)=
2(\sum_{i=1}^sl_{i}\vec{z}_i+2\vec{d})$, which implies $l_i=0$ for $1\leq i\leq s$. Hence $\pi(\frac{p_1}{2}\vec{x}_1)=\pi(\frac{p_2}{2}\vec{x}_2)=
\pi(\frac{p_3}{2}\vec{x}_3)=2\vec{d}$.
\end{proof}
\subsection{The associated matrix of $\pi$}

Recall that $\pi\colon \bbL(\mathbf{p})\rightarrow \bbL(\mathbf{q})$ is an admissible homomorphism, and $\pi(\vec{x}_i)=\sum_{j=1}^s a_{ij}\vec{z}_j+a_{i,s+1}\vec{d}$ is in normal form for each $1\leq i\leq t$.
Let
$$A=\begin{bmatrix}
a_{11}&a_{12} & \cdots&a_{1s}\\
a_{21}&a_{22} & \cdots&a_{2s}\\
\cdots&\cdots&\cdots&\cdots\\
a_{t1}&a_{t2}&\cdots&a_{ts}
\end{bmatrix},\;
\beta=\begin{bmatrix}
a_{1,s+1}\\
a_{2,s+1}\\
\cdots\\
a_{t,s+1}
\end{bmatrix},$$
and call
$\widehat{A}=\big[A \;\,\beta\big]$
the \emph{associated matrix} of $\pi$.

In this subsection, we study the properties of the associated matrix $\widehat{A}$. Since $\pi$ is admissible, we have $a_{i,s+1}\geq 0$ for $1\leq i\leq t$, i.e. $\beta\geq 0$. Moreover,

\begin{lem} \label{one part is zero}
For $1\leq i\leq t$, $a_{i,s+1}\neq 0$ if and only if $(a_{i1}, a_{i2}, \cdots, a_{is})=0$. In this case, $a_{i,s+1}=1$ or 2.
\end{lem}

\begin{proof}
  If $a_{i,s+1}=0$, then $(a_{i1}, a_{i2}, \cdots, a_{is})\neq 0$ since $\pi(\vec{x}_i)=\sum_{j=1}^sa_{ij}\vec{z}_{j}+a_{i,s+1}\vec{d}\neq 0$. Now assume $a_{i,s+1}>0$, then ${\rm mult}(\pi(\vec{x}_i))=a_{i,s+1}+1\geq 2$. Hence there exists $0\neq\vec{y}\in\ker\pi$ with ${\rm mult}(\vec{x}_i+\vec{y})>0$, i.e. $\vec{x}_i+\vec{y}\geq 0$. By Lemma \ref{pi c} we have $\vec{y}=l_j\vec{x}_j+(p_i-1)\vec{x}_i-\vec{c}=l_j\vec{x}_j-\vec{x}_i$ for some $j$ with $1\leq l_j< p_j$. Then it follows from Proposition \ref{ker pi} that $\pi(\vec{x}_i)=\vec{d}$ or $2\vec{d}$. Hence $(a_{i1}, a_{i2}, \cdots, a_{is})=0$ and $a_{i,s+1}=1$ or 2.
 \end{proof}

Recall that ${\bf p}=(p_1, p_2, \cdots, p_t)$ and $\mathbf{q}=(q_1,q_2, \cdots, q_s)$. The following technical result plays an important role for further discussion.

\begin{prop} \label{how to determine the map pi}

(1) If $\ker\pi=\langle \frac{p_1}{n}\vec{x}_1-\frac{p_2}{n}\vec{x}_2\rangle$, let $b=|\{i\in\{1,2\}\;|\; p_i>n\}|$. Then
$s=n(t-2)+b$ and
$$\mu(\pi(\vec{x}_i))=
  \left\{
  \begin{array}{lll}
    0&& i\in\{1,2\} {\rm \;and\;} p_i=n;\\
   1&&  i\in\{1,2\} {\rm \;and\;} p_i>n\,;\\
   n&& i\not\in\{1,2\}.
  \end{array}
\right.$$
(2) If $\ker\pi=\langle \frac{p_1}{2}\vec{x}_1-\frac{p_2}{2}\vec{x}_2, \frac{p_1}{2}\vec{x}_1-\frac{p_3}{2}\vec{x}_3\rangle$, denote by $b=|\{i\in\{1,2,3\}\;|\; p_i>2\}|$. Then
$s=4(t-3)+2b$ and $$\mu(\pi(\vec{x}_i))=
  \left\{
  \begin{array}{lll}
    0&& i\in\{1,2,3\} {\rm \;and\;} p_i=2;\\
   2&&  i\in\{1,2,3\} {\rm \;and\;} p_i>2\,;\\
   4&& i\not\in\{1,2,3\}.
  \end{array}
\right.$$
Consequently, each column of $A$ has only one nonzero entry 1.
\end{prop}

\begin{proof}
(1) By Proposition \ref{ker pi}, we have $\pi(\frac{p_1}{n}\vec{x}_1)=\pi(\frac{p_2}{n}\vec{x}_2)=\vec{d}$.
Then for $i\in\{1,2\}$, $\mu(\pi(\vec{x}_i))=0$ or 1 according to $p_i=n$ or $p_i>n$ respectively.
We claim that $a_{3, s+1}=a_{4,s+1}=\cdots=a_{t,s+1}=0$. For contradiction, we assume $a_{i, s+1}\neq 0$ for some $3\leq i\leq t$. Then by Lemma \ref{one part is zero}, we have $\pi(\vec{x}_i)=a_{i, s+1}\vec{d}$, which implies that $\vec{x}_i-a_{i, s+1}\cdot\frac{p_1}{n}\vec{x}_1\in \ker\pi$, a contradiction.
Hence $\pi(\vec{x}_i)=\sum_{j=1}^sa_{ij}\vec{z}_j$ for $3\leq i\leq t$.

Recall that $|\ker\pi|=n$. Then by Lemma \ref{pi c} $\pi(\vec{c})=n\vec{d}$, which ensures that $\mu(\pi(\vec{x}_i))\leq n$ for $3\leq i\leq t$. It follows that there are at most $n(t-2)+b$ many non-zero entries in $A$.
Moreover, $\im\pi$ is effective forces that each column of $A$ admits non-zero entries. It means that at least $s$-many non-zero entries should appear in $A$. We claim that $s=n(t-2)+b$, which implies there are exactly $s$-many non-zero entries in $A$. It follows that each column of $A$ has exactly one non-zero entry and $\mu(\pi(\vec{x}_i))=n$ for $3\leq i\leq t$.

Now we prove the claim $s=n(t-2)+b$. For this we let $\vec{x}'=t\vec{c}-\sum_{i=1}^t\vec{x}_i$. Then we have ${\rm mult}(\vec{x}')=1$ and ${\rm mult}(\vec{x}'+\vec{y})=2$ for any $0\neq \vec{y}\in\ker\pi$. We consider the following three cases:
\begin{itemize}
  \item[(1.1)]  $b=0$; in this case, let $\vec{x}'':=\sum_{i=3}^t\vec{x}_i$. We know that ${\rm mult}(\vec{x}'')=1$ and ${\rm mult}(\vec{x}''+\vec{y})=0$ for any $0\neq \vec{y}\in\ker\pi$. Then $\pi$ is admissible implies
${\rm mult}(\pi(\vec{x}''))=1$. Note that $\pi(\vec{x}'')=\sum_{i=3}^t\sum_{j=1}^sa_{ij}\vec{z}_j=
\sum_{j=1}^s(\sum_{i=3}^ta_{ij})\vec{z}_j$. It follows that $0<\sum_{i=3}^ta_{ij}<q_j$ for $1\leq j\leq s$.
Hence $\pi(\vec{x}')=\pi(t\vec{c})-\pi(\vec{x}_1)-\pi(\vec{x}_2)-\pi(\vec{x}'')
=\sum_{j=1}^s(q_j-\sum_{i=3}^ta_{ij})\vec{z}_j+(nt-2-s)\vec{d}$. Therefore, $1+2(n-1)=\sum_{\vec{x}\in \pi^{-1}(\pi(\vec{x}'))} {\rm mult}(\vec{x})={\rm mult}(\pi(\vec{x}'))=nt-1-s.$ It follows that $s=n(t-2).$
  \item[(1.2)] $b=1$; in this case, we assume $\frac{p_1}{n}=1$ and $\frac{p_2}{n}>1$ without loss of generality. Let $\vec{x}'':=\sum_{i=2}^t\vec{x}_i$. Then by similar arguments as in (1.1) we have $1+2(n-1)=\sum_{\vec{x}\in \pi^{-1}(\pi(\vec{x}'))} {\rm mult}(\vec{x})={\rm mult}(\pi(\vec{x}'))=nt-s.$ It follows that $s=n(t-2)+1.$
  \item[(1.3)] $b=2$; in this case, let $\vec{x}'':=\sum_{i=1}^t\vec{x}_i$. Then by similar arguments as above we have $1+2(n-1)=\sum_{\vec{x}\in \pi^{-1}(\pi(\vec{x}'))} {\rm mult}(\vec{x})={\rm mult}(\pi(\vec{x}'))=nt-s+1.$ It follows that $s=n(t-2)+2.$
\end{itemize}

(2) By Proposition \ref{ker pi}, for $1\leq i\leq 3$ we have $\pi(\frac{p_i}{2}\vec{x}_i)=2\vec{d}$, hence
$\mu(\pi(\vec{x}_i))\leq 2$. Observe that $|\ker\pi|=4$. We claim $s=4(t-3)+2b$. For this we let $\vec{x}'=t\vec{c}-\sum_{i=1}^t\vec{x}_i$ again. Then ${\rm mult}(\vec{x}')=1$ and ${\rm mult}(\vec{x}'+\vec{y})=2$ for any $0\neq \vec{y}\in\ker\pi$. Hence ${\rm mult}(\pi(\vec{x}'))=7$.
Observe that $0\leq b\leq 3$. Without loss of generality, we assume $p_i=2$ for $1\leq i\leq 3-b$ and $p_i>2$ for $4-b\leq i\leq 3$. By analysing $\vec{x}''=\sum_{i=4-b}^t\vec{x}_i$ as in (1.1) we get ${\rm mult}(\pi(\vec{x}'))={\rm mult}\big(\pi(t\vec{c})-\sum_{i=1}^{3-b}\pi(\vec{x}_i)-\pi(\vec{x}'')\big)=4t-2(3-b)-s+1=7$. Thus $s=4(t-3)+2b$.

By similar arguments as in (1) we obtain that $\mu(\pi(\vec{x}_i))=0$ or 2 according to $p_i=2$ or $p_i>2$ respectively, and $\mu(\pi(\vec{x}_i))=4$ for $4\leq i\leq t$, moreover, there exists exactly one non-zero entry in each column of $A$.

Finally, to finish the proof we still need to show that each non-zero entry of $A$ equals 1. Recall that $a_{ij}$'s appear in the expression $\pi(\vec{x}_i)=\sum_{j=1}^s a_{ij}\vec{z}_j+a_{i,s+1}\vec{d}$. We discuss on $\mu(\pi(\vec{x}_i))$ for each $i$. In fact, if $\mu(\pi(\vec{x}_i))=|\ker\pi|=n$, then
there exist $j_1, j_2,\cdots, j_n$ such that $a_{ij_k}\neq 0$ for $1\leq k\leq n$ and  $\pi(\vec{x}_i)=\sum_{k=1}^n a_{ij_k}\vec{z}_{j_k}$. By Lemma \ref{pi c}, $n\vec{d}=\pi(\vec{c})=\pi(p_i\vec{x}_i)=\sum_{k=1}^n p_ia_{ij_k}\vec{z}_{j_k}$. It follows that $q_{j_k}=p_ia_{ij_k}$ for each $k$. Then by Proposition \ref{coprime}, we get $a_{ij_k}={\rm g.c.d.} (a_{ij_k}, q_{j_k})=1$. Now assume $\mu(\pi(\vec{x}_i))\neq|\ker\pi|$. Then $\mu(\pi(\vec{x}_i))=0,1$
or 2 by the above discussion. We consider the following three cases.

\begin{itemize}
  \item [(i)] If $\mu(\pi(\vec{x}_i))=0$, there is nothing to prove.
  \item [(ii)] If $\mu(\pi(\vec{x}_i))=1$, then $\pi(\vec{x}_i)=a_{ij}\vec{z}_j$ and $\ker\pi=\langle \frac{p_i}{n}\vec{x}_i-\frac{p_k}{n}\vec{x}_k\rangle$ for some $j$ and $k$. By Proposition
      \ref{ker pi},  $\pi(\frac{p_i}{n}\vec{x}_i)=\frac{p_i}{n}\cdot a_{ij}\vec{z}_j=\vec{d}$. Hence $q_j=\frac{p_i}{n}\cdot a_{ij}$. Then by Proposition \ref{coprime}, we get $a_{ij}={\rm g.c.d.} (a_{ij}, q_j)=1$.
  \item [(iii)] If $\mu(\pi(\vec{x}_i))=2$, then $\pi(\vec{x}_i)=a_{ij_1}\vec{z}_{j_1}+a_{ij_2}\vec{z}_{j_2}$ and $\ker\pi=\langle\frac{p_i}{2}\vec{x}_i-\frac{p_{k_1}}{2}\vec{x}_{k_1}, \\ \frac{p_i}{2}\vec{x}_i-\frac{p_{k_2}}{2}\vec{x}_{k_2}\rangle$ for some $j_1, j_2, k_1$ and $k_2$. By Proposition \ref{ker pi},  we have $\pi(\frac{p_i}{2}\vec{x}_i)=
      \frac{p_i}{2}(a_{ij_1}\vec{z}_{j_1}+a_{ij_2}\vec{z}_{j_2})
      =2\vec{d}$. It follows that $q_{j_1}=\frac{p_i}{2}a_{ij_1}$ and $q_{j_2}=\frac{p_i}{2}a_{ij_2}$. Then by Proposition \ref{coprime}, we get $a_{ij_1}=a_{ij_2}=1$.
\end{itemize}
\end{proof}

\begin{rem}\label{explicit expression of pi}
(1) According to the proof of Proposition \ref{how to determine the map pi}, the associated matrix $\widehat{A}=\big[A \;\,\beta\big]$ has one of the following shapes:

\begin{itemize}
  \item [(i)] if $\ker\pi$ is of Cyclic type, then
$$\widehat{A}=\begin{bmatrix}
f_{11}&&&&& f_{12}\\
&f_{21}&&&& f_{22}\\
&&{e_n}&&&0\\
&&&\ddots&&\vdots\\
&&&&{e_n}&0
\end{bmatrix},$$
where ${e_n}=(\underbrace{1\quad 1\quad\cdots\quad 1}_{n\ \text{times}})$, $(f_{i1}; f_{i2})=(\emptyset;1)$ or $(1;0)$ for $1\leq i\leq 2$;

  \item [(ii)] if $\ker\pi$ is of Klein type, then
$$\widehat{A}=\begin{bmatrix}
f_{11}\quad f_{12}&&&&&&f_{13}\\
&f_{21}\quad f_{22}&&&&&f_{23}\\
&&f_{31}\quad f_{32}&&&&f_{33}\\
&&&{1\quad 1\quad 1\quad 1}&&&0\\
&&&&\ddots&&\vdots\\
&&&&&{1\quad 1\quad 1\quad 1}&0
\end{bmatrix},$$
where $(f_{i1}, f_{i2}; f_{i3})=(\emptyset,\emptyset;2)$ or $(1,1;0)$ for $1\leq i\leq 3$.
\end{itemize}

(2) We can rewrite the generators of $\bbL({\bf q})$ by $\vec{z}_{i_j}$'s for $1\leq i\leq t$, such that the admissible homomorphism $\pi: \bbL({\bf p})\to \bbL({\bf q})$ can be described explicitly as follows:

\begin{itemize}
  \item[(i)] if $\ker\pi$ is of Cyclic type, then
\begin{itemize}
  \item [-] for $1\leq i\leq 2$, $\pi(\vec{x}_i)=
  \left\{
  \begin{array}{ll}
    \vec{z}_{i_1}& {\rm \;if\;} p_i>n;\\
   \vec{d}&    {\rm \;if\;} p_i=n;
  \end{array}
\right.$
  \item [-] for $3\leq i\leq t$, $\pi(\vec{x}_i)=\sum\limits_{j=1}^{n}\vec{z}_{i_j}$;
\end{itemize}

\item[(ii)] if $\ker\pi$ is of Klein type, then
\begin{itemize}
  \item [-] for $1\leq i\leq 3$, $\pi(\vec{x}_i)=
  \left\{
  \begin{array}{ll}
    \vec{z}_{i_1}+\vec{z}_{i_2}& {\rm \;if\;} p_{i}> 2;\\
   2\vec{d}&    {\rm \;if\;} p_i=2;
  \end{array}
\right.$
  \item [-] for $4\leq i\leq t$, $\pi(\vec{x}_i)=\sum\limits_{j=1}^{4}\vec{z}_{i_j}$;
\end{itemize}
\end{itemize}
\noindent where the $\vec{z}_{i_j}$'s are all pairwise distinct generators of $\bbL({\bf q})$ in both cases.
\end{rem}
\subsection{$\pi$ preserves the dualizing elements}

We denote by $\vec{\omega}$ and $\vec{\omega}'$ the dualizing elements of $\bbL(\mathbf{p})$ and $\bbL(\mathbf{q})$ respectively. Recall that $\vec{\omega}=(t-2)\vec{c}-\sum_{i=1}^t\vec{x}_i$ and $\vec{\omega}'=(s-2)\vec{d}-\sum_{i=1}^s\vec{z}_i$.
We have the following result:

\begin{prop} \label{pi omega}
Let $\pi: \bbL(\mathbf{p})\to \bbL(\mathbf{q})$ be an admissible homomorphism. Then $\pi(\vec{\omega})=\vec{\omega}^{\prime}$. Consequently,
$\bbL(\mathbf{p})$ is of domestic, tubular or wild type if and only if so is $\bbL(\mathbf{q})$.
\end{prop}

\begin{proof}
By Proposition \ref{how to determine the map pi}, we have $\pi(\sum_{i=1}^t\vec{x}_i)=\sum_{j=1}^s\vec{z}_j+r\vec{d}$ for some integer $r$. Thus $\pi(\vec{\omega})=\pi((t-2)\vec{c}-\sum_{i=1}^t\vec{x}_i)=
((t-2)n-r)\vec{d}-\sum_{j=1}^s\vec{z}_j=\sum_{j=1}^s(q_j-1)\vec{z}_j+a\vec{d}$ for some integer $a$. By the proof of Lemma \ref{pi c} (2), we know that $a=-2$. Hence $\pi(\vec{\omega})=\sum_{j=1}^s(q_j-1)\vec{z}_j-2\vec{d}=(s-2)\vec{d}-\sum_{i=1}^s\vec{z}_i=\vec{\omega}^{\prime}$.

Now observe that $p\vec{\omega}=\delta(\vec{\omega})\vec{c}$, which yields $p\vec{\omega}^{\prime}=\pi(p\vec{\omega})=\pi(\delta(\vec{\omega})\vec{c})
=\delta(\vec{\omega})\cdot n\vec{d}$. It follows that $p\delta(\vec{\omega}^{\prime})=\delta(\vec{\omega})\cdot n\delta(\vec{d})=nq\delta(\vec{\omega})$.
Therefore, $\delta(\vec{\omega})$ and $\delta(\vec{\omega}^{\prime})$
have the same sign since all of $p,q,n$ are positive. This finishes the proof.
\end{proof}

\subsection{Derive $\bbL(\mathbf{q})$ from $\bbL(\mathbf{p})$}

In Proposition \ref{ker pi} we have shown that $\ker\pi$ of an admissible homomorphism $\pi\colon \bbL(\mathbf{p})\rightarrow \bbL(\mathbf{q})$ is either of Cyclic type or Klein type. In this subsection, we will show that any such kind subgroup of $\bbL(\mathbf{p})$ can be realized as a kernel of certain admissible homomorphism. Moreover, we will determine all the string group $\bbL(\mathbf{q})$'s when $\bbL(\mathbf{p})$ is given. Observe that by adding new entry 1's in a sequence $\mathbf{q}$ will not change the isoclass of the associated string group $\bbL(\mathbf{q})$, more precisely, assume $\tilde{\bf q}=(1, \cdots, 1, q_1, q_2,\cdots, q_s)$, then $\bbL(\tilde{\bf q})\cong\bbL(\bf q)$. For our convenience, we will assume $q_i\geq 1$ from now on.

\begin{prop}\label{from subgroup to admissible homomorphism}
Let $H$ be a subgroup of $\bbL(\mathbf{p})$ of Cyclic type or Klein type. Then there exist a unique weight sequence ${\bf q}$ (up to permutation) together with an admissible homomorphism $\pi\colon \bbL(\mathbf{p})\rightarrow \bbL(\mathbf{q})$, such that $\ker\pi=H$. More precisely,
\begin{itemize}
 \item[(1)] if $H$ is of Cyclic type, then
     $\mathbf{q}=(\frac{p_1}{n}, \frac{p_2}{n}, \underbrace{p_3, \cdots, p_3}_{n\ \text{times}}, \cdots, \underbrace{p_t, \cdots, p_t}_{n\ \text{times}})$;
 \item[(2)] if $H$ is of Klein type, then
     $$\mathbf{q}=(\frac{p_1}{2}, \frac{p_1}{2}, \frac{p_2}{2}, \frac{p_2}{2},
     \frac{p_3}{2},\frac{p_3}{2}, \underbrace{p_4, \cdots, p_4}_{4\ \text{times}},\cdots, \underbrace{p_t, \cdots, p_t}_{4\ \text{times}}).$$
\end{itemize}
\end{prop}

\begin{proof}
(1) Assume $H$ is of Cyclic type, i.e. $H=\langle \frac{p_1}{n}\vec{x}_1-\frac{p_2}{n}\vec{x}_2\rangle$. Let $$\mathbf{q}=(\frac{p_1}{n}, \frac{p_2}{n}, \underbrace{p_3, \cdots, p_3}_{n\ \text{times}}, \cdots, \underbrace{p_t, \cdots, p_t}_{n\ \text{times}}).$$ For convenience, the associated generators of $\bbL({\bf q})$ are denoted by $\vec{z}_{1_1}, \vec{z}_{2_1}$, and $\vec{z}_{i_j}$'s for $3\leq i\leq t$ and $1\leq j\leq n$ respectively. The following assignments $$\pi(\vec{x}_1)=\vec{z}_{1_1},\; \pi(\vec{x}_2)=\vec{z}_{2_1},\; \pi(\vec{x}_i)=\sum\limits_{j=1}^{n}\vec{z}_{i_j}\; (3\leq i\leq t)$$ define a homomorphism $\pi\colon \bbL(\mathbf{p})\rightarrow \bbL(\mathbf{q})$ with $\ker\pi=\langle \frac{p_1}{n}\vec{x}_1-\frac{p_2}{n}\vec{x}_2\rangle$.

Obviously, $\pi$ is effective. Moreover, for any $\vec{z}\in\im\pi$, $\vec{z}$ has the form
$$\vec{z}=a_1\vec{z}_{1_1}+a_2\vec{z}_{2_1}+\sum_{i=3}^ta_i (\sum_{j=1}^n\vec{z_{i_j}})+a\vec{d},$$ where $0\leq a_i<\frac{p_i}{n}\; (1\leq i\leq 2$), $0\leq a_i<p_i\; (3\leq i\leq t)$ and $a\in\bbZ$. If $a<0$, then ${\rm mult}(\vec{z})=0=\sum_{\vec{x}\in\pi^{-1}(\vec{z})}{\rm mult}(\vec{x})$. Now assume $a\geq 0$ and let $\vec{x}=\sum_{i=1}^ta_i \vec{x}_{i}+a\cdot \frac{p_1}{n}\vec{x}_1$. Then $\pi(\vec{x})=\vec{z}$ and $\sum_{\vec{h}\in\ker\pi}{\rm mult}(\vec{x}+\vec{h})=a+1={\rm mult}(\vec{z})$. It follows that $\pi$ is admissible.

(2) Assume $H$ is of Klein type. Then by similar arguments as above, we show that the assignments $$\pi(\vec{x}_1)=\vec{z}_{1_1}+\vec{z}_{1_2},\; \pi(\vec{x}_2)=\vec{z}_{2_1}+\vec{z}_{2_2},\;\pi(\vec{x}_3)=\vec{z}_{3_1}+\vec{z}_{3_2},\; \pi(\vec{x}_i)=\sum\limits_{j=1}^{4}\vec{z}_{i_j}\; (4\leq i\leq t)$$ define an admissible homomorphism $\pi\colon \bbL(\mathbf{p})\rightarrow \bbL(\mathbf{q})$ with $\ker\pi=H$.

Moreover, according to Remark \ref{explicit expression of pi}, the associated matrix of an admissible homomorphism is determined by its kernel. It follows that the weight sequence ${\bf q}$ is unique up to permutation.
\end{proof}

\subsection{Composition and decomposition}

In this subsection we investigate the composition and the decomposition of admissible homomorphisms. Observe that the following result holds for group homomorphisms (not necessary admissible).
\begin{lem} \label{lem:4.2}
Assume $\pi_1\colon \bbL(\mathbf{p})\rightarrow \bbL(\mathbf{r})$ and $\pi_2\colon \bbL(\mathbf{r})\rightarrow \bbL(\mathbf{q})$ are group homomorphisms. Then the following conditions are equivalent:
\begin{itemize}
   \item[(1)] $\ker\pi_2=\pi_1(\ker(\pi_2\pi_1))$;
   \item[(2)] $|\ker(\pi_2\pi_1)|=|\ker\pi_2|\cdot|\ker\pi_1|$.
 \end{itemize}
\end{lem}

\begin{proof} The homomorphism $\pi_1$ restricts to a homomorphism $\pi_1|_{\ker(\pi_2\pi_1)}: \ker(\pi_2\pi_1)\to \ker\pi_2$ with kernel equal to $\ker(\pi_2\pi_1)\bigcap\ker\pi_1=\ker\pi_1$. Hence, we have the following exact sequence (where $\iota$ is the inclusion map):
$$\xymatrix{0 \ar[r] & \ker\pi_1 \ar[rr]^{\iota}&&\ker(\pi_2\pi_1) \ar[rr]^{\pi_1|_{\ker(\pi_2\pi_1)}}&&  \ker\pi_2.}$$
Then both statements (1) and (2) are equivalent to that $\pi_1|_{\ker(\pi_2\pi_1)}$ is surjective.
We are done.
\end{proof}

The following two propositions deal with the composition and decomposition of admissible homomorphisms respectively.

\begin{prop} \label{Prop:3.13}
Assume $\pi_1\colon \bbL(\mathbf{p})\rightarrow \bbL(\mathbf{r})$ and $\pi_2\colon \bbL(\mathbf{r})\rightarrow \bbL(\mathbf{q})$ are admissible homomorphisms. Then the following conditions are equivalent:
\begin{itemize}
   \item[(1)] $\pi_2\pi_1$ is admissible;
   \item[(2)] $|\ker(\pi_2\pi_1)|=|\ker\pi_2|\cdot|\ker\pi_1|$.
  \end{itemize}
\end{prop}
\begin{proof} (1)$\Rightarrow$(2):
Assume $|\ker\pi_1|=m$ and $|\ker\pi_2|=n$. By Lemma \ref{pi c}
we have $\pi_1(\vec{c})=m\vec{d}'$ and $\pi_2(\vec{d}')=n\vec{d}$, where $\vec{d}'$ is the canonical element of $\bbL(\mathbf{r})$. Thus $(\pi_2\pi_1)(\vec{c})=mn\vec{d}$. Since $\pi_2\pi_1$ is admissible, we get $|\ker(\pi_2\pi_1)|=mn=|\ker\pi_2|\cdot|\ker\pi_1|$.

(2)$\Rightarrow$(1):
For $1\leq i\leq 2$, let $\widehat{A_i}=\big[A_i \;\,\beta_i\big]$ be the associate matrix of $\pi_{i}$. Then the associate matrix of $\pi_2\pi_{1}$ has the form $\big[
A_1A_2\;\, \beta'\big]$ for some $\beta'$.
By Proposition \ref{how to determine the map pi}, each column of $A_i$ has only one nonzero entry 1, then so does $A_1A_2$. Hence $\pi_2\pi_1$ is effective.
Moreover, for each $\vec{z}\in \im(\pi_2\pi_1)$, we have $${\rm mult}(\vec{z})=\sum_{\vec{y}\in \pi^{-1}_{2}(\vec{z})} {\rm mult}(\vec{y})=\sum_{\vec{y}\in \pi^{-1}_{2}(\vec{z})}\sum_{\vec{x}\in {\pi_{1}^{-1}}(\vec{y})} {\rm mult}(\vec{x})=\sum_{\vec{x}\in {(\pi_{2}\pi_{1})}^{-1}(\vec{z})} {\rm mult}(\vec{x}),$$ where for the last equation we use the condition (2). Hence $\pi_2\pi_1$ is admissible.
\end{proof}

\begin{prop}\label{decomposition}
Assume $\pi\colon \bbL(\mathbf{p})\rightarrow \bbL(\mathbf{q})$ is an admissible homomorphism with $\ker\pi=H_1\times H_2$. Then there exist a weight sequence $\mathbf{r}$ together with admissible homomorphisms $\pi_1\colon \bbL(\mathbf{p})\rightarrow \bbL(\mathbf{r})$ and $\pi_2\colon \bbL(\mathbf{r})\rightarrow \bbL(\mathbf{q})$, such that $\pi=\pi_2\pi_1$, $\ker\pi_1=H_1$ and $\ker\pi_2=\pi_1(\ker\pi)\cong H_2$.
\end{prop}

\begin{proof}
If $H_i$ is trivial for $i=1$ or 2, then there is nothing to show, hence we will assume $H_i$ is non-trivial in the following. Recall that $\mathbf{p}=(p_1,p_2,\cdots,p_t)$. By Proposition \ref{ker pi}, $\ker\pi$ has one of the following forms (up to permutation).

(1) $\ker\pi=\langle \frac{p_1}{n}\vec{x}_1-\frac{p_2}{n}\vec{x}_2\rangle$;
then by Proposition \ref{from subgroup to admissible homomorphism}, we know that $$\mathbf{q}=(\frac{p_1}{n}, \frac{p_2}{n}, \underbrace{p_3, \cdots, p_3}_{n\ \text{times}}, \cdots \underbrace{p_t, \cdots, p_t}_{n\ \text{times}}),$$
and $\pi\colon \bbL(\mathbf{p})\rightarrow \bbL(\mathbf{q})$ is given by:
$$\begin{array}{lllll}
&\pi: & \vec{x}_1\mapsto \vec{z}_{1_1}; & \vec{x}_2\mapsto \vec{z}_{2_1}; & \vec{x}_i\mapsto \sum\limits_{j=1}^{n}\vec{z}_{i_j}\;\; (3\leq i\leq t).
\end{array}$$
Assume $|H_i|=n_i$ for $i=1,2$, then $H_i=\langle\frac{p_1}{n_i}\vec{x}_1-\frac{p_2}{n_i}\vec{x}_2\rangle$, and we have $n=n_1\cdot n_2$ and ${\rm g.c.d.}(n_1, n_2)=1$.
Let
$\mathbf{r}=(\frac{p_1}{n_1},\frac{p_2}{n_1},\underbrace{p_3, \cdots, p_3}_{n_1\ \text{times}},\cdots,\underbrace{p_t, \cdots, p_t}_{n_1\ \text{times}})$. We denote the generators of $\bbL(\mathbf{r})$ by $\vec{y}_{1_1}, \vec{y}_{2_1}$ and $\vec{y}_{i_j}$'s, which satisfy the relations $\frac{p_1}{n_1}\vec{y}_{1_1}=\frac{p_2}{n_1}\vec{y}_{2_1}=p_i\vec{y}_{i_j}$ for $3\leq i\leq t$ and $1\leq j\leq n_1$.
Define group homomorphisms $\pi_1\colon \bbL(\mathbf{p})\rightarrow \bbL(\mathbf{r})$ and $\pi_2\colon \bbL(\mathbf{r})\rightarrow \bbL(\mathbf{q})$ on generators as follows :
$$\begin{array}{lllll}
&\pi_1: & \vec{x}_1\mapsto \vec{y}_{1_1}; & \vec{x}_2\mapsto \vec{y}_{2_1}; & \vec{x}_i\mapsto \sum\limits_{j=1}^{n_1}\vec{y}_{i_j}; \\
&\pi_2: & \vec{y}_{1_1}\mapsto \vec{z}_{1_1}; & \vec{y}_{2_1}\mapsto \vec{z}_{2_1}; & \vec{y}_{i_k}\mapsto \sum\limits_{j=(k-1)n_2+1}^{kn_2}\vec{z}_{i_j};
\end{array}$$
where $3\leq i\leq t,\;1\leq k\leq n_1$.
It is easy to check that $\pi_1$ and $\pi_2$ are both admissible and $\pi=\pi_2\pi_1$. Moreover, $\ker\pi_1=\langle \frac{p_1}{n_1}\vec{x}_1-\frac{p_2}{n_1}\vec{x}_2\rangle=H_1$ and $\ker\pi_2=\langle \frac{p_1}{n}\vec{y}_{1_1}-\frac{p_2}{n}\vec{y}_{2_1}\rangle=\pi_1(\ker(\pi_2\pi_1))=\pi_1(\ker\pi)$, which is a cyclic group of order $n_2$, hence isomorphic to $H_2$.

(2) $\ker\pi=\langle \frac{p_1}{2}\vec{x}_1-\frac{p_2}{2}\vec{x}_2, \frac{p_1}{2}\vec{x}_1-\frac{p_3}{2}\vec{x}_3\rangle$; in this case, we can assume $H_1=\langle \frac{p_1}{2}\vec{x}_1-\frac{p_2}{2}\vec{x}_2\rangle$ and $H_2=\langle \frac{p_1}{2}\vec{x}_1-\frac{p_3}{2}\vec{x}_3\rangle$ without loss of generality.
By Proposition \ref{from subgroup to admissible homomorphism} we have $$\mathbf{q}=(\frac{p_1}{2}, \frac{p_1}{2}, \frac{p_2}{2}, \frac{p_2}{2},
     \frac{p_3}{2},\frac{p_3}{2}, \underbrace{p_4, \cdots, p_4}_{4\ \text{times}},\cdots\underbrace{p_t, \cdots, p_t}_{4\ \text{times}}),$$
and $\pi\colon \bbL(\mathbf{p})\rightarrow \bbL(\mathbf{q})$ is given by:
$$\begin{array}{llll}
&\pi: & \vec{x}_i\mapsto \vec{z}_{i_1}+\vec{z}_{i_2}\;\;(1\leq i\leq 3); & \vec{x}_i\mapsto \sum\limits_{j=1}^{4}\vec{z}_{i_j}\;\;(4\leq i\leq t).
\end{array}$$
Let $\mathbf{r}=(\frac{p_1}{2},\frac{p_2}{2},p_3,p_3,p_4,p_4,\cdots,p_t,p_t)$ and denote the generators of $\bbL(\mathbf{r})$ by $\vec{y}_{1_1}, \vec{y}_{2_1}$ and $\vec{y}_{i_j}$ ($3\leq i\leq t, \,1\leq j\leq 2$) similarly as above.
Define group homomorphisms $\pi_1\colon \bbL(\mathbf{p})\rightarrow \bbL(\mathbf{r})$ and $\pi_2\colon \bbL(\mathbf{r})\rightarrow \bbL(\mathbf{q})$ on generators as follows :
$$\begin{array}{lllll}
&\pi_1: & \vec{x}_1\mapsto \vec{y}_{1_1}; & \vec{x}_2\mapsto \vec{y}_{2_1}; & \vec{x}_i\mapsto \vec{y}_{i_1}+\vec{y}_{i_2}; \\
&\pi_2: & \vec{y}_{k_1}\mapsto \vec{z}_{k_1}+\vec{z}_{k_2}; & \vec{y}_{k_2}\mapsto \vec{z}_{k_3}+\vec{z}_{k_4}; & \vec{y}_{3_j}\mapsto \vec{z}_{3_j};
\end{array}$$
where $3\leq i\leq t,\;1\leq k\neq 3\leq t,\;1\leq j\leq 2.$ It is easy to check that $\pi_1$ and $\pi_2$ are both admissible and $\pi=\pi_2\pi_1$. Moreover, $\ker\pi_1=\langle \frac{p_1}{2}\vec{x}_1-\frac{p_2}{2}\vec{x}_2\rangle=H_1$ and $\ker\pi_2=\langle \frac{p_3}{2}\vec{y}_{3_1}-\frac{p_3}{2}\vec{y}_{3_2}\rangle=\pi_1(\ker(\pi_2\pi_1))=\pi_1(\ker\pi)\cong H_2.$
\end{proof}

\subsection{Classification of admissible homomorphisms}\label{subsection of classification of adm}

According to Proposition \ref{from subgroup to admissible homomorphism} and Remark \ref{explicit expression of pi}, we know that an admissible homomorphism $\pi: \bbL(\mathbf{p})\to \bbL(\mathbf{q})$ is determined by its kernel, which is a finite subgroup of the torsion group $t\bbL(\mathbf{p})$. Basing on this fact, we can produce a procedure to determine all the admissible homomorphisms $\pi: \bbL(\mathbf{p})\to \bbL(\mathbf{q})$ when $\bbL(\mathbf{p})$ is given. More precisely, the procedure will include the following three steps:

\emph{Step 1: } describe all the subgroup of the torsion group $t\bbL(\mathbf{p})$ of Cyclic type or Klein type (this can be easily done since $t\bbL(\mathbf{p})$ is finite);

\emph{Step 2: }  derive $\bbL(\mathbf{q})$ (up to permutation) from $\bbL(\mathbf{p})$ for each subgroup obtained in Step 1 (this follows from Proposition \ref{from subgroup to admissible homomorphism});

\emph{Step 3: }  determine the map $\pi$ on generators (this follows from Remark
    \ref{explicit expression of pi}).

Using this procedure, we can classify all the admissible homomorphisms for weighted projective lines of domestic and tubular types as follows.

\begin{thm} \label{classification} All the admissible homomorphisms $\pi: \bbL(\bf p)\to \bbL(\bf q)$ for domestic types or tubular types are given as below (up to permutation).

\begin{table}[ht]\label{table 1}
\begin{tabular}{|c|c|c|c|c|}
  \hline
 $\mathbf{p}$&$\mathbf{q}$&$\ker\pi$&$\pi(\vec{x}_1,\vec{x}_2,\cdots,\vec{x}_t)$\\
  \hline
 $(2,3,4)$&$(2,3,3)$&$\langle\vec{x}_1-2\vec{x}_3\rangle$&$(\vec{d},\vec{z}_2+\vec{z}_3,\vec{z}_1)$\\
  \hline
 $(2,3,3)$&$(2,2,2)$&$\langle\vec{x}_2-\vec{x}_3\rangle$&$(\vec{z}_1+\vec{z}_2+\vec{z}_3,\vec{d},\vec{d})$\\
  \hline
  $(2,2,2p_{3})$&$(2,2,p_3)$&$\langle\vec{x}_2-p_3\vec{x}_3\rangle$&$(\vec{z}_1+\vec{z}_2,\vec{d},\vec{z}_3)$\\
  \hline
  $(2,2,2p_3)$&$(p_3,p_3)$&$\langle\vec{x}_1-\vec{x}_2,\vec{x}_1-p_3\vec{x}_3\rangle$&$(2\vec{d},2\vec{d},\vec{z}_1+\vec{z}_2)$\\
  \hline
  $(2,2,p_3)$&$(p_3,p_3)$&$\langle\vec{x}_1-\vec{x}_2\rangle$&$(\vec{d},\vec{d},\vec{z}_1+\vec{z}_2)$\\
  \hline
  $(nq_1,nq_2)$&$(q_1,q_2)$&$\langle q_1\vec{x}_1-q_2\vec{x}_2\rangle$&$(\vec{z}_1,\vec{z}_2)$\\
  \hline
  $(nq,n)$&$(q)$&$\langle q\vec{x}_1-\vec{x}_2\rangle$&$(\vec{z}_1,\vec{d})$\\
  \hline
  $(n,n)$&$()$&$\langle \vec{x}_1-\vec{x}_2\rangle$&$(\vec{d},\vec{d})$\\
  \hline
  \end{tabular}
  \vspace{1em}
  \caption{The list of admissible homomorphisms between domestic types}\label{table for domestic admissible}
\end{table}

\bibliography{}

\begin{table}[ht]
\begin{tabular}{|c|c|c|c|}
\hline
$\mathbf{p}$&$\mathbf{q}$&$\ker\pi$&$\pi(\vec{x}_1,\vec{x}_2,\cdots,\vec{x}_t)$\\
\hline
\multirow{2}*{$(2,2,2,2)$}&\multirow{2}*{$(2,2,2,2)$}&$\langle \vec{x}_1-\vec{x}_2\rangle$&$(\vec{d},\vec{d},\vec{z}_1+\vec{z}_2,\vec{z}_3+\vec{z}_4)$\\
\cline{3-4}
&&$\langle \vec{x}_1-\vec{x}_2,\vec{x}_1-\vec{x}_3\rangle$&$(2\vec{d},2\vec{d},2\vec{d},\vec{z}_1+\vec{z}_2+\vec{z}_3+\vec{z}_4)$\\
\hline
\multirow{4}*{$(4,4,2)$}&$(4,4,2)$&$\langle 2\vec{x}_1-\vec{x}_3\rangle$&$(\vec{z}_3,\vec{z}_1+\vec{z}_2,\vec{d})$\\
\cline{2-4}
&\multirow{3}*{$(2,2,2,2)$}&$\langle \vec{x}_1-\vec{x}_2\rangle$&$(\vec{d},\vec{d},\vec{z}_1+\vec{z}_2+\vec{z}_3+\vec{z}_4)$\\
\cline{3-4}
&&$\langle 2\vec{x}_1-2\vec{x}_2\rangle$&$(\vec{z}_1,\vec{z}_2,\vec{z}_3+\vec{z}_4)$\\
\cline{3-4}
&&$\langle 2\vec{x}_1-2\vec{x}_2,2\vec{x}_1-\vec{x}_3\rangle$&$(\vec{z}_1+\vec{z}_2,\vec{z}_3+\vec{z}_4,2\vec{d})$\\
\hline
$(3,3,3)$&$(3,3,3)$&$\langle \vec{x}_1-\vec{x}_2\rangle$&$(\vec{d},\vec{d},\vec{z}_1+\vec{z}_2+\vec{z}_3)$\\
\hline
\multirow{2}*{$(6,3,2)$}&$(3,3,3)$&$\langle 3\vec{x}_1-\vec{x}_3\rangle$&$(\vec{z}_1,\vec{z}_2+\vec{z}_3,\vec{d})$\\
\cline{2-4}
&$(2,2,2,2)$&$\langle 2\vec{x}_1-\vec{x}_2\rangle$&$(\vec{z}_1,\vec{d},\vec{z}_2+\vec{z}_3+\vec{z}_4)$\\
\hline
\end{tabular}
\vspace{1em}
\caption{The list of admissible homomorphisms between tubular types}\label{table for tubular admissible}
\end{table}

\end{thm}

We remark that our procedure is also valid for weighted projective lines of wild types. In the following we give a typical example to show how it works.

\begin{exm}\label{example for 46710} Let $\pi: \bbL(\mathbf{p})\to \bbL(\mathbf{q})$ be an admissible homomorphism with $\mathbf{p}=(4,6,7,10)$. Since the torsion group $t\bbL({\bf p})=\{0,2\vec{x}_1-3\vec{x}_2, 2\vec{x}_1-5\vec{x}_4, 3\vec{x}_2-5\vec{x}_4\}$,
by Proposition \ref{ker pi}, $\ker\pi\,(\neq 0)$ has one of the following forms:
$$(i)\, \langle 2\vec{x}_1-3\vec{x}_2 \rangle,\;(ii)\, \langle 2\vec{x}_1-5\vec{x}_4 \rangle,\; (iii)\, \langle 3\vec{x}_2-5\vec{x}_4 \rangle,\; (iv)\, \langle 2\vec{x}_1-3\vec{x}_2, 2\vec{x}_1-5\vec{x}_4 \rangle.$$
Now by using Proposition \ref{from subgroup to admissible homomorphism}, we can derive the weight type $\mathbf{q}$ respectively:
$$\begin{array}{lll}
(i)\,(2,3,7,7,10,10), &&(ii)\, (2,6,6,7,7,5), \\
(iii)\, (4,4,3,7,7,5), &&(iv)\, (2,2,3,3,7,7,7,7,5,5).
\end{array}$$
Finally, by applying Remark \ref{explicit expression of pi} we can determine the homomorphism $\pi$ by giving $\pi(\vec{x}_1, \vec{x}_2,\vec{x}_3,\vec{x}_4):=(\pi(\vec{x}_1), \pi(\vec{x}_2), \pi(\vec{x}_3), \pi(\vec{x}_4))$ as follows:
$$\begin{array}{ll}
(i)\,(\vec{z}_1,\vec{z}_2,\vec{z}_{3_1}+\vec{z}_{3_2},\vec{z}_{4_1}+\vec{z}_{4_2}), \;
&(ii)\,(\vec{z}_1,\vec{z}_{2_1}+\vec{z}_{2_2},\vec{z}_{3_1}+\vec{z}_{3_2},\vec{z}_4),\\
(iii)\, (\vec{z}_{1_1}+\vec{z}_{1_2},\vec{z}_2,\vec{z}_{3_1}+\vec{z}_{3_2},\vec{z}_4),\;
&(iv)\, (\vec{z}_{1_1}+\vec{z}_{1_2},\vec{z}_{2_1}+\vec{z}_{2_2},
\sum_{j=1}^4\vec{z}_{3_j},
\vec{z}_{4_1}+\vec{z}_{4_2}).
\end{array}
$$
\end{exm}

\section{Equivariant equivalences induced by degree-shift actions}

Two abelian categories $\mathcal{A}$ and $\mathcal{B}$ are said to be \emph{equivariant equivalent} if there exist a strict action of a finite group $G$ on $\mathcal{A}$ such that the equivariant category $\mathcal{A}^{G}$ is equivalent to $\mathcal{B}$.
In this section, we study the equivariant equivalence relations induced by degree-shift actions between the categories of coherent sheaves on weighted projective lines.

\subsection{Compatible algebra homomorphisms}

Let $\mathbf{p}=(p_1,p_2,\cdots,p_t)$ and $\mathbf{q}=(q_1,q_2,\cdots,q_s)$ be two weight sequences.
Denote the associated string groups by $$\bbL({\bf p})=\bbZ\{\vec{x}_1, \vec{x}_2,\cdots, \vec{x}_t\}/(p_1\vec{x}_1=p_2\vec{x}_2=\cdots = p_t\vec{x}_t:=\vec{c});$$
 $$\bbL({\bf q})=\bbZ\{\vec{z}_1, \vec{z}_2,\cdots, \vec{z}_s\}/(q_1\vec{z}_1=q_2\vec{z}_2=\cdots = q_s\vec{z}_s:=\vec{d}).$$ Let ${\boldsymbol\lambda}=(\lambda_1, \lambda_2, \cdots, \lambda_t)$ and ${\boldsymbol\mu}=(\mu_1, \mu_2, \cdots, \mu_s)$ be two normalized parameter sequences. The weighted projective lines $\bbX({\bf p}; {\boldsymbol\lambda})$ and  $\bbX({\bf q}; {\boldsymbol\mu})$ admit homogeneous coordinate algebras $S(\mathbf{p}; {\boldsymbol\lambda})$ and $S(\mathbf{q}; {\boldsymbol\mu})$ respectively, having the following forms:
$$S(\mathbf{p}; {\boldsymbol\lambda})=\mathbf{k}[X_1, X_2,\cdots, X_t]/\big(X_i^{p_i}-(X_2^{p_2}-\lambda_{i}X_1^{p_1})|3\leq i\leq t\big)\triangleq \mathbf{k}[x_1, x_2,\cdots, x_t],$$
$$S(\mathbf{q}; {\boldsymbol\mu})=\mathbf{k}[Z_1,Z_2,\cdots, Z_s]/\big(Z_i^{q_i}-(Z_2^{q_2}-\mu_{i}Z_1^{q_1})|3\leq i\leq s\big)\triangleq \mathbf{k}[z_1, z_2,\cdots, z_s].$$

Let $\pi\colon \bbL(\mathbf{p})\rightarrow \bbL(\mathbf{q})$ be an admissible homomorphism. Recall that an algebra homomorphisms $\phi: S(\mathbf{p}; {\boldsymbol\lambda})\rightarrow S(\mathbf{q}; {\boldsymbol\mu})$ is \emph{compatible} with
$\pi$ if $\phi(S(\mathbf{p}; {\boldsymbol\lambda})_{\vec{x}})\subseteq S(\mathbf{q}; {\boldsymbol\mu})_{\pi(\vec{x})}$ for each $\vec{x}\in \bbL(\mathbf{p})$.

\begin{lem}\label{Prop:4.1}
Assume $\phi: S(\mathbf{p}; {\boldsymbol\lambda})\rightarrow S(\mathbf{q}; {\boldsymbol\mu})$ is an algebra homomorphism, then $\phi$ is compatible with $\pi$
if and only if $\phi(S(\mathbf{p}; {\boldsymbol\lambda})_{\vec{x}_i})\subseteq S(\mathbf{q}; {\boldsymbol\mu})_{\pi(\vec{x}_i)}$ for $1\leq i\leq t$.
\end{lem}

\begin{proof}
 The ``only if"  part is obvious. We only prove the``if " part.
 For each $\vec{x}\in \bbL(\mathbf{p})$, write $\vec{x}=\sum_{i=1}^t l_{i}\vec{x}_i+l\vec{c}$ in normal form. Then  $S(\mathbf{p}; {\boldsymbol\lambda})_{\vec{x}}$ has a $\mathbf{k}$-basis $\{x^{ap_1}_1x^{bp_2}_2\prod_{i=1}^tx^{l_i}_i\,|\,a+b=l, a\geq 0, b\geq 0\}$, and  $$\phi(x^{ap_1}_1x^{bp_2}_2\prod_{i=1}^tx^{l_i}_i)=
 \phi(x_1)^{ap_1}\cdot\phi(x_2)^{bp_2}\cdot\prod_{i=1}^t\phi(x_i)^{l_i},$$ which has degree
 $\sum_{i=1}^t l_{i}\pi(\vec{x}_i)+(a+b)\pi(\vec{c})=\pi(\vec{x})$ by assumption. Then we are done.
\end{proof}

Let $\pi\colon \bbL(\mathbf{p})\rightarrow \bbL(\mathbf{q})$ be an admissible homomorphism. By Proposition \ref{ker pi}, if $\ker\pi\neq 0$, then it is either of Cyclic type or Klein type.
Assume $\phi: S(\mathbf{p}; {\boldsymbol\lambda})\rightarrow S(\mathbf{q}; {\boldsymbol\mu})$ is an algebra homomorphism which is compatible with $\pi$. Recall that $\phi$ induces an $\im\pi$-graded algebra homomorphism
$$\bar{\phi}\colon \pi_*S(\mathbf{p}; {\boldsymbol\lambda})\longrightarrow S(\mathbf{q}; {\boldsymbol\mu})_{\im\pi}.$$
For convenience we denote by $\bar{\phi}_{\vec{z}}\colon (\pi_*S(\mathbf{p}; {\boldsymbol\lambda}))_{\vec{z}}\longrightarrow S(\mathbf{q}; {\boldsymbol\mu})_{\vec{z}}$ the restriction map of $\bar{\phi}$ on each degree $\vec{z}\in\im\pi$.

The following result is useful and will be used frequently later on.

\begin{prop}\label{Prop:4.2} Keep the notation as above. If one of the following holds:
\begin{itemize}
  \item [(1)] $\ker\pi$ is of Cyclic type and $\bar{\phi}_{\vec{d}}$ is surjective;
  \item [(2)] $\ker\pi$ is of Klein type and $\bar{\phi}_{2\vec{d}}$ is surjective;
\end{itemize}
then $\bar{\phi}$ is a surjective, which yields an equivalence
\begin{equation}\label{equivalence induced by ker pi}({\rm coh}\mbox{-}\mathbb{X}(\mathbf{p}; {\boldsymbol\lambda}))^{\ker\pi}\stackrel{\sim}\longrightarrow {\rm coh}\mbox{-}\mathbb{X}(\mathbf{q}; {\boldsymbol\mu}).\end{equation}
\end{prop}

\begin{proof}
(1) If $\ker\pi$ is of Cyclic type, we have $\ker\pi=\langle \frac{p_1}{n}\vec{x}_1-\frac{p_2}{n}\vec{x}_2\rangle$ with $|\ker\pi|=n$. By Remak \ref{explicit expression of pi}, for $1\leq i\leq 2$ we have $\pi(\vec{x}_i)=\vec{z}_{i_1}$ or $\vec{d}$ according to $p_i>n$ or $p_i=n$ respectively; and for $3\leq i\leq t$, we have $\pi(\vec{x}_i)=\sum_{j=1}^{n}\vec{z}_{i_j}$; where all the $\vec{z}_{i_j}$'s are pairwise distinct generators of $\bbL(\mathbf{q})$.

For any $\vec{z}\in \im\pi$, write $\vec{z}=l_1\vec{z}_{1_1}+l_2\vec{z}_{2_1}+\sum_{i=3}^{t}\sum_{j=1}^{n}l_i\vec{z}_{i_j}+l\vec{d}$ in normal form, then $S(\mathbf{q}; {\boldsymbol\mu})_{\vec{z}}=z_{1_1}^{l_1}z_{2_1}^{l_2}\cdot\prod_{i=3}^{t}(\prod_{j=1}^{n}z_{i_j})^{l_i}\cdot S(\mathbf{q}; {\boldsymbol\mu})_{l\vec{d}}$. Recall that $\dim_{\bf k}S(\mathbf{q}; {\boldsymbol\mu})_{\vec{d}}=2$. Let $\{u,v\}$ be a basis of $S(\mathbf{q}; {\boldsymbol\mu})_{\vec{d}}$, then $\{u^{a}v^{b}\; |\; a+b=l, a, b\geq 0\}$ is a basis of $S(\mathbf{q}; {\boldsymbol\mu})_{l\vec{d}}$. Assume $\bar{\phi}_{\vec{d}}$ is surjective, then so is $\bar{\phi}_{l\vec{d}}$. Moreover, $\phi$ is compatible with $\pi$ implies $z_{1_1}^{l_1}z_{2_1}^{l_2}\cdot\prod_{i=3}^{t}(\prod_{j=1}^{n}z_{i_j})^{l_i}=\prod_{i=1}^t \phi(x_i)^{l_i}=\phi(\prod_{i=1}^t x_i^{l_i})$ up to a scalar. Therefore, $\bar{\phi}_{\vec{z}}$ is surjective for any $\vec{z}\in \im\pi$, and then $\bar{\phi}$ is a surjective.

(2) If $\ker\pi$ is of Klein type, we have $\ker\pi=\langle \frac{p_1}{2}\vec{x}_1-\frac{p_2}{2}\vec{x}_2, \frac{p_1}{2}\vec{x}_1-\frac{p_3}{2}\vec{x}_3\rangle$. By Remak \ref{explicit expression of pi}, for $1\leq i\leq 3$ we have $\pi(\vec{x}_i)=\vec{z}_{i_1}+\vec{z}_{i_2}$ or $2\vec{d}$ according to $p_i\neq 2$ or $p_i=2$ respectively; and for $4\leq i\leq t$ we have $\pi(\vec{x}_i)=\sum_{j=1}^{4}\vec{z}_{i_j}$; where all the $\vec{z}_{i_j}$'s are pairwise distinct generators of $\bbL(\mathbf{q})$.

For any $\vec{z}\in \im\pi$, write $\vec{z}=\sum_{i=1}^{3}l_{i}(\vec{z}_{i_1}+\vec{z}_{i_2})+ \sum_{i=4}^{t}\sum_{j=1}^{4}l_i\vec{z}_{i_j}+2l\vec{d}$ in normal form, then $S(\mathbf{q}; {\boldsymbol\mu})_{\vec{z}}=\prod_{i=1}^{3}(z_{i_1}z_{i_2})^{l_i}\cdot\prod_{i=4}^{t}(\prod_{j=1}^{4}z_{i_j})^{l_i}\cdot S(\mathbf{q}; {\boldsymbol\mu})_{2l\vec{d}}$. Recall that $\dim_{\bf k}S(\mathbf{q}; {\boldsymbol\mu})_{2\vec{d}}=3$ and $\{u^2,v^2,uv\}$ is a basis of $S(\mathbf{q}; {\boldsymbol\mu})_{2\vec{d}}$. Hence $\{u^{a}v^{b}\; |\; a+b=2l, a, b\geq 0\}$ is a basis of $S(\mathbf{q}; {\boldsymbol\mu})_{2l\vec{d}}$. Observe that $u^{a}v^{b}=u^{2a_1}v^{2b_1}(uv)^{c}$, where
$a_1=[\frac{a}{2}]$, $b_1=[\frac{b}{2}]$, and $c=0$ or 1 according to $a$ is even or odd respectively, here $[x]$ denotes the integer part of $x$.
Hence $\bar{\phi}_{2\vec{d}}$ is surjective implies $\bar{\phi}_{2l\vec{d}}$ is surjective, too. Moreover, $\phi$ is compatible with $\pi$ implies $\prod_{i=1}^{3}(z_{i_1}z_{i_2})^{l_i}\cdot\prod_{i=4}^{t}(\prod_{j=1}^{4}z_{i_j})^{l_i}=\prod_{i=1}^t \phi(x_i)^{l_i}=\phi(\prod_{i=1}^t x_i^{l_i})$ up to a scalar. Therefore, $\bar{\phi}_{\vec{z}}$ is surjective for any $\vec{z}\in \im\pi$, and then $\bar{\phi}$ is a surjective.

Now by Proposition \ref{Prop:2.3}, the surjective map $\bar{\phi}$ induces an equivalence (\ref{equivalence induced by ker pi}).
\end{proof}

Let $\pi_1, \pi_2$ be homomorphisms of groups and $\phi_1, \phi_2$ be homomorphisms of algebras as follows: $$\xymatrix{\bbL(\mathbf{p})\ar[r]^{\pi_1}& \bbL(\mathbf{r})\ar[r]^{\pi_2}&\bbL(\mathbf{q})}; \quad \xymatrix{S(\mathbf{p}; {\boldsymbol\lambda})\ar[r]^{\phi_1}& \ S(\mathbf{r}; {\boldsymbol\tau})\ar[r]^{\phi_2}&\ S(\mathbf{q}; {\boldsymbol\mu}).}$$
If $\phi_i$ is compatible with $\pi_i$ for $i=1,2$, then $\phi_2\phi_1$ is also compatible with $\pi_2\pi_1$ by
$$(\phi_2\phi_1)(S(\mathbf{p}; {\boldsymbol\lambda})_{\vec{x}})\subseteq \phi_2(S(\mathbf{r}; {\boldsymbol\tau})_{\pi_1(\vec{x})})\subseteq S(\mathbf{q}; {\boldsymbol\mu})_{(\pi_2\pi_1)(\vec{x})},\quad \vec{x}\in \bbL(\mathbf{p}).$$
This induces an $\im (\pi_2\pi_1)$-graded homomorphism $$\overline{\phi_2\phi_1}\colon (\pi_2\pi_1){_*}S(\mathbf{p}; {\boldsymbol\lambda})\longrightarrow S(\mathbf{q}; {\boldsymbol\mu})_{\im(\pi_2\pi_1)}.$$

\begin{prop}\label{decomposition of category}
Keep the notation as above. Assume the following hold:
\begin{itemize}
  \item [(1)] $\pi_1, \pi_2$ are admissible and $|\ker(\pi_2\pi_1)|=|\ker\pi_2|\cdot|\ker\pi_1|$;
  \item [(2)] $\phi_i$ is compatible with $\pi_i$ and $\bar{\phi_i}$ is surjective for $1\leq i\leq 2$;
\end{itemize}
then $\overline{\phi_2\phi_1}$ is surjective, which induces an equivalence
\begin{equation}\label{equiv decomposition}({\rm coh}\mbox{-}\mathbb{X}(\mathbf{p}; {\boldsymbol\lambda}))^{\ker(\pi_2\pi_1)}\stackrel{\sim}\longrightarrow {\rm coh}\mbox{-}\mathbb{X}(\mathbf{q}; {\boldsymbol\mu}).
\end{equation}
\end{prop}

\begin{proof}
Firstly we show that $\overline{\phi_2\phi_1}$ is surjective.
That is, for each $\vec{z}\in \im(\pi_2\pi_1)$ and $f\in S(\mathbf{q}; {\boldsymbol\mu})_{\vec{z}}$, we need to find an element $h\in S(\mathbf{p}; {\boldsymbol\lambda})$, such that $\overline{\phi_2\phi_1}(h)=f$.

Observe that $\im(\pi_2\pi_1)\subseteq \im\pi_2$, and $\bar{\phi_2}: (\pi_2)_\ast S(\mathbf{r}; {\boldsymbol\tau})\to S(\mathbf{q}; {\boldsymbol\mu})_{\im\pi_2}$ is a surjective homomorphism between $\im\pi_2$-graded algebras. Hence there exists some $g\in((\pi_2){_*}S(\mathbf{r}; {\boldsymbol\tau}))_{\vec{z}}$ satisfying $\bar{\phi_2}(g)=f$.
Since $\vec{z}\in \im(\pi_2\pi_1)$, there
exists $\vec{x}\in \bbL(\mathbf{p})$ such that $\vec{z}= \pi_2\pi_1(\vec{x})$. Hence $\pi_1(\vec{x})\in \pi_2^{-1}(\vec{z})$ and then $\pi_2^{-1}(\vec{z})=\pi_1(\vec{x})+\ker\pi_2$.
By Lemma \ref{lem:4.2} we have $\ker\pi_2\subseteq\im\pi_1$, hence $\pi_2^{-1}(\vec{z})\subseteq \im\pi_{1}.$ Therefore, $$g\in ((\pi_2){_*}S(\mathbf{r}; {\boldsymbol\tau}))_{\vec{z}}=\bigoplus_{\vec{y}\in \pi_2^{-1}(\vec{z})} S(\mathbf{r}; {\boldsymbol\tau})_{\vec{y}}\subseteq S(\mathbf{r}; {\boldsymbol\tau})_{\im\pi_1}.$$
Now $\bar{\phi_1}: (\pi_1)_\ast S(\mathbf{p}; {\boldsymbol\lambda})\to S(\mathbf{r}; {\boldsymbol\tau})_{\im\pi_1}$ is surjective implies that there exists some  $h\in S(\mathbf{p}; {\boldsymbol\lambda})$ such that $\bar{\phi_1}(h)=g$. Therefore, $\overline{\phi_2\phi_1}(h)=\bar{\phi_2}(g)=f$.

By Proposition \ref{Prop:3.13}, $\pi_2\pi_1$ is admissible, and $\phi_2\phi_1$ is compatible with $\pi_2\pi_1$. Then by Proposition \ref{Prop:2.3}, the homomorphism $\overline{\phi_2\phi_1}$ induces an equivalence
(\ref{equiv decomposition}).
\end{proof}

\subsection{From admissible homomorphisms to equivariant equivalences}

In this subsection, we show that each admissible homomorphism between string groups induces an equivariant equivalence between the associated categories of coherent sheaves.

First we consider the admissible homomorphisms induced by permutations. Recall that any permutation $\sigma$ on $\{1,2,\cdots,t\}$ induces an admissible homomorphism $\pi_{\sigma}: \bbL({\bf p})\to \bbL({\bf p})$ via $\pi_{\sigma}(\vec{x}_i)=\vec{x}_{\sigma(i)}$ for $1\leq i\leq t$. In particular, if
$\sigma$ is a transposition $(i,j)$ for some $1\leq i<j\leq t$, then we have the following result:
\begin{lem}\label{lemma for permutation ij}
For any parameter sequence ${\boldsymbol\lambda}$, there exist a parameter sequence ${\boldsymbol\mu}$ and an algebra homomorphism $\phi_{(i,j)}: S(\mathbf{p}; {\boldsymbol\lambda})\to S(\mathbf{p}; {\boldsymbol\mu})$, which is compatible with $\pi_{(i,j)}$ and induces an equivalence
\begin{equation}\label{middle step for ker0}{\rm coh}\mbox{-}\mathbb{X}(\mathbf{p}; {\boldsymbol\lambda})\stackrel{\sim}\longrightarrow {\rm coh}\mbox{-}\mathbb{X}(\mathbf{p};{\boldsymbol\mu}).
\end{equation}
\end{lem}

\begin{proof}
Denote by $\mathbb{S}:=\{i,j\}\cap \{1,2\}$.
It is routine to check that the assignments in each case below define an algebra homomorphism $\tilde{\phi}_{(i,j)}: S(\mathbf{p}; {\boldsymbol\lambda})\to S(\mathbf{p}; \tilde{{\boldsymbol\mu}})$.
 \begin{itemize}
   \item [(1)] $\mathbb{S}=\emptyset$; let $\tilde{\mu}_i=\lambda_j$, $\tilde{\mu}_j=\lambda_i$ and $\tilde{\mu}_k=\lambda_k$ for any $k\notin \{i,j\}$, and define
$$\tilde{\phi}_{(i,j)}:\;\; x_i\mapsto z_j,\;\;x_j\mapsto z_i,\;\;x_k\mapsto z_k\; (k\notin \{i,j\});$$
\item [(2)] $\mathbb{S}=\{1, 2\}$; let $\tilde{\mu}_1=\lambda_1$, $\tilde{\mu}_2=\lambda_2$, and $\tilde{\mu}_k=\frac{1}{\lambda_k}$ for $k\notin\{1,2\}$, and define
$$\tilde{\phi}_{(i,j)}:\;\; x_1\mapsto z_2,\;\;x_2\mapsto z_1,\;\;x_k\mapsto (-\lambda_k)^{\frac{1}{p_k}}z_k\; (k\notin \{1,2\});$$
  \item [(3)] $\mathbb{S}=\{2\}$; let $\tilde{\mu}_1=\lambda_1$, $\tilde{\mu}_2=\lambda_2$, $\tilde{\mu}_i=-\lambda_i$, and $\tilde{\mu}_k=\lambda_k-\lambda_i$ for $k\notin\{1,2,i\}$, and define
$$\tilde{\phi}_{(i,j)}:\;\; x_2\mapsto z_i,\;\;x_i\mapsto z_2,\;\;x_k\mapsto z_k\; (k\notin \{i,j\});$$
 \item[(4)] $\mathbb{S}=\{1\}$; let $\tilde{\mu}_k=\lambda_k$ for $k\in\{1,2,i\}$ and $\tilde{\mu}_k=\frac{\lambda_k\lambda_i}{\lambda_k-\lambda_i}$ otherwise, and define $\tilde{\phi}_{(i,j)}$ as follows for  $k\notin \{1,2,i\}$:
$$x_1\mapsto ({\lambda_i})^{-\frac{1}{p_1}}z_i,\;\;x_2\mapsto z_2,\;\;x_i\mapsto ({\lambda_i})^{\frac{1}{p_i}}z_1, \;\;x_k\mapsto ({\frac{\lambda_i-\lambda_k}{\lambda_i}})^{\frac{1}{p_k}}z_k.$$
\end{itemize}
Observe that $\tilde{\mu}_1=\lambda_1$ and $\tilde{\mu}_2=\lambda_2$ in each case above, but $\tilde{\mu}_3\neq \lambda_3$ in general. We let ${\boldsymbol\mu}=(\tilde{\mu}_1,\; \tilde{\mu}_2,\; 1, \; \frac{\tilde{\mu}_4}{\tilde{\mu}_3},\;\cdots,\frac{\tilde{\mu}_t}{\tilde{\mu}_3}).$ Then ${\boldsymbol\mu}$ is normalized and the following assignments
$$z_1\mapsto (\tilde{\mu}_3)^{-\frac{1}{p_1}}z_1,\;\;z_k\mapsto z_k\; (2\leq k\leq t)$$
define an algebra homomorphism $\tilde{\phi}:=S(\mathbf{p}; \tilde{\boldsymbol\mu})\to S(\mathbf{p}; {\boldsymbol\mu})$.

Denote by $\phi_{(i,j)}:=\tilde{\phi}\circ\tilde{\phi}_{(i,j)}$. Then by Lemma \ref{Prop:4.1}, $\phi_{(i,j)}$ is compatible with $\pi_{(i,j)}$. Moreover, by similar arguments as in the proof of \cite[Proposition 3.8]{CC17}, we obtain that $\phi_{(i,j)}$ induces a surjective homomorphism $\overline{\phi_{(i,j)}}\colon \pi_*S(\mathbf{p}; {\boldsymbol\lambda})\rightarrow S(\mathbf{p}; {\boldsymbol\mu})_{\im\pi}$, which yields an equivalence (\ref{middle step for ker0}) by Proposition \ref{Prop:2.3}.
\end{proof}

Now we can state the main result of this subsection.

\begin{thm}\label{from admissible to equivariant}
Let $\pi\colon \bbL(\mathbf{p})\rightarrow \bbL(\mathbf{q})$ be an admissible homomorphism. Then for any parameter sequence ${\boldsymbol\lambda}$, there exist a parameter sequence ${\boldsymbol\mu}$ and an algebra homomorphism $\phi: S(\mathbf{p}; {\boldsymbol\lambda})\to S(\mathbf{q}; {\boldsymbol\mu})$, which is compatible with $\pi$ and induces an equivalence
\begin{equation}\label{equiv induced by admissible}({\rm coh}\mbox{-}\mathbb{X}(\mathbf{p}; {\boldsymbol\lambda}))^{\ker\pi}\stackrel{\sim}\longrightarrow {\rm coh}\mbox{-}\mathbb{X}(\mathbf{q}; {\boldsymbol\mu}).
\end{equation}
\end{thm}

\begin{proof}
First we consider $\ker\pi=0$. By Lemma \ref{cor:3.5}, this happens if and only if $\pi=\pi_{\sigma}$ for some permutation $\sigma$ on $\{1,2,\cdots,t\}$. It is well known that $\sigma$ expresses as a product of transposition $(i,j)$'s, accordingly, $\pi_{\sigma}$ expresses as a composition of $\pi_{(i,j)}$'s. According to Lemma \ref{lemma for permutation ij}, each $\pi_{(i,j)}$ induces an equivalence (\ref{equiv induced by admissible}), then so does $\pi_{\sigma}$ by using Proposition \ref{decomposition of category}.

In the following we assume $\ker\pi\neq 0$.
By Proposition \ref{ker pi}, $\ker\pi$ is either of Cyclic type or Klein type (up to permutation). Hence we consider the following two cases:

(1) $\ker\pi$ is of Cyclic type, i.e. $\ker\pi=\langle \frac{p_1}{n}\vec{x}_1-\frac{p_2}{n}\vec{x}_2\rangle$. Firstly we assume $p_i>n$ for $i=1,2$ and $t\geq 3$. By Proposition \ref{from subgroup to admissible homomorphism} we have $\mathbf{q}=(\frac{p_1}{n}, \frac{p_2}{n}, \underbrace{p_3, \cdots, p_3}_{n\ \text{times}}, \cdots \underbrace{p_t, \cdots, p_t}_{n\ \text{times}}),$ and the generators of $\bbL(\mathbf{q})$ can be rewritten as $\vec{z}_{1_1}, \vec{z}_{2_1}$ and $\vec{z}_{i_j}$'s ($3\leq i\leq t, \,1\leq j\leq n$) respectively by Remak \ref{explicit expression of pi}. Moreover, $\pi(\vec{x}_1)=\vec{z}_{1_1}$, $\pi(\vec{x}_2)=\vec{z}_{2_1}$ and $\pi(\vec{x}_i)=\sum_{j=1}^{n}\vec{z}_{i_j}$ for $3\leq i\leq t$.

We denote by $z_{1_1}^{\frac{p_1}{n}}:=x$ and $z_{2_1}^{\frac{p_2}{n}}:=y$. Then for $3\leq i\leq t$, $$y^{n}-\lambda_ix^{n}=(y-\mu_{i_1}x)\cdots(y-\mu_{i_n}x),$$ where $\mu_{i_j}$'s are pairwise distinct $n$-th roots of $\lambda_i$. In particular, for $\lambda_3=1$, we can choose $\mu_{3_1}=1$. Now let $${\boldsymbol\mu}=(\infty, 0,\; \mu_{3_1}, \mu_{3_2},\cdots, \mu_{3_n},\;\cdots, \;\mu_{t_1}, \mu_{t_2}, \cdots, \mu_{t_{n}}).$$
Then in the homogeneous coordinate algebra $S(\mathbf{q}; {\boldsymbol\mu})$, we have $$z_{i_j}^{p_i}=z_{2_1}^{\frac{p_2}{n}}-\mu_{i_j}z_{1_1}^{\frac{p_1}{n}}=y-\mu_{i_j}x; \quad 3\leq i\leq t,\; 1\leq j\leq n.$$
In particular, $$\prod_{j=1}^{n}z_{i_j}^{p_i}=(y-\mu_{i_1}x)\cdots(y-\mu_{i_n}x)=y^{n}-\lambda_ix^{n}
=z_{2_1}^{p_2}-\lambda_iz_{1_1}^{p_1}.$$
We define an algebra homomorphism $\phi': \mathbf{k}[X_1,X_2,\cdots,X_t]\to S(\mathbf{q}; {\boldsymbol\mu})$ on generators as follows: $$\phi'(X_1)=z_{1_1},\phi'(X_2)=z_{2_1},\phi'(X_i)=\prod_{j=1}^{n}z_{i_j}\;\;(3\leq i\leq t).$$
Then $$\phi'(X_{i}^{p_i}-X_{2}^{p_2}+\lambda_iX_{1}^{p_1})
=\phi'(X_{i})^{p_i}-\phi'(X_{2})^{p_2}+\lambda_i\phi'(X_{1})^{p_1}
=\prod_{j=1}^{n}z_{i_j}^{p_i}-z_{2_1}^{p_2}+\lambda_iz_{1_1}^{p_1}=0.$$ Hence $\phi'$ induces an algebra homomorphism $\phi: S(\mathbf{p}; {\boldsymbol\lambda})\to S(\mathbf{q}; {\boldsymbol\mu})$ as follows:
$$\phi(x_1)=z_{1_1},\phi(x_2)=z_{2_1},\phi(x_i)=\prod_{j=1}^{n}z_{i_j}(3\leq i\leq t).$$
Clearly, $\phi$ is compatible with $\pi$, hence yields a homomorphism between $\im\pi$-graded algebras
$\bar{\phi}\colon \pi_*S({\bf p}; {\boldsymbol\lambda})\to S({\bf q}; {\boldsymbol\mu})_{\im\pi}.$
We claim $\bar{\phi}$ is surjective. By Proposition \ref{Prop:4.2}, it suffices to show that $\bar{\phi}_{\vec{d}}$ is surjective. This follows from the fact that $\phi(x_1^{\frac{p_1}{n}})=x$ and $\phi(x_2^{\frac{p_2}{n}})=y$ span the $\mathbf{k}$-space $S({\bf q}; {\boldsymbol\mu})_{\vec{d}}$.
Now by Proposition \ref{Prop:2.3}, $\bar{\phi}$ induces an equivalence (\ref{equiv induced by admissible}).

If $p_1=n$ or $p_2=n$, we can prove (\ref{equiv induced by admissible}) similarly.
Now we assume $t\leq 2$. Then by Proposition \ref{pi omega}, both of $\mathbb{X}(\mathbf{p})$ and $\mathbb{X}(\mathbf{q})$ are of domestic types. Then according to Theorem \ref{classification}, we have the following cases (for some $n>1$):
\begin{itemize}
   \item [(i)] $\mathbf{p}=(nq_1,nq_2),\;\mathbf{q}=(q_1,q_2)$ with $q_1>1,\;q_2>1$;
   \item [(ii)] $\mathbf{p}=(nq,n), \;\mathbf{q}=(q)$ with $q>1$;
   \item [(iii)] $\mathbf{p}=(n,n), \;\mathbf{q}=()$.
 \end{itemize}
In each case, we define an admissible homomorphism $\pi: \bbL(\mathbf{p})\to \bbL(\mathbf{q})$ and an algebra homomorphism $\phi: S(\mathbf{p})\to S(\mathbf{q})$ on generators as below.

 \begin{table}[h]\label{table 12}
\begin{tabular}{|c|c|c|c|c|}
  \hline
 $\mathbf{p}$&$\mathbf{q}$&$\pi(\vec{x}_1,\vec{x}_2)$ &$\phi(x_1, x_2)$\\
  \hline
  $(nq_1,nq_2)$&$(q_1,q_2)$&$(\vec{z}_1,\vec{z}_2)$&$(z_1,z_2)$\\
  \hline
  $(nq,n)$&$(q)$&$(\vec{z}_1, \vec{d})$&$(z_1,z_2)$\\
  \hline
  $(n,n)$&$()$&$(\vec{d},\vec{d})$&$(z_1+z_2,z_1-z_2)$\\
  \hline
  \end{tabular}
\end{table}

\noindent It is easy to see that in each case $\phi$ is compatible with $\pi$, which induces an algebra homomorphism $\bar{\phi}: \pi_*S(\bf p)\to S(\bf q)_{\im\pi}$. By using Proposition \ref{Prop:4.2} it is easy to check $\bar{\phi}$ is surjective, which yields an equivalence (\ref{equiv induced by admissible}) by Proposition \ref{Prop:2.3}.

(2) $\ker\pi$ is of Klein type, i.e. $\ker\pi=\langle \frac{p_1}{2}\vec{x}_1-\frac{p_2}{2}\vec{x}_2, \frac{p_1}{2}\vec{x}_1-\frac{p_3}{2}\vec{x}_3\rangle$.
By Proposition \ref{decomposition}, we can decompose $\pi=\pi_2\pi_1: \xymatrix{\bbL(\mathbf{p})\ar[r]^{\pi_1}& \bbL(\mathbf{r})\ar[r]^{\pi_2}&\bbL(\mathbf{q})}$, where $\pi_1$ and $\pi_2$ are admissible homomorphisms and $|\ker\pi_1|=|\ker\pi_2|=2$. By (1), there exists a parameter sequence ${\boldsymbol\tau}$ and an algebra homomorphism $\phi_1: S(\mathbf{p}; {\boldsymbol\lambda})\to S(\mathbf{r}; {\boldsymbol\tau})$, which is compatible with $\pi_1$ and induces an $\im\pi_1$-graded surjective homomorphism $\bar{\phi_1}$. Similarly, there exists a parameter sequence ${\boldsymbol\mu}$ and an algebra homomorphism $\phi_2: S(\mathbf{r}; {\boldsymbol\tau})\to S(\mathbf{q}; {\boldsymbol\mu})$, which is compatible with $\pi_2$ and induces an $\im\pi_2$-graded surjective homomorphism $\bar{\phi_2}$.

By Proposition \ref{decomposition of category}, the algebra homomorphism $\phi_2\phi_1\colon S(\mathbf{p}; {\boldsymbol\lambda})\rightarrow S(\mathbf{q}; {\boldsymbol\mu})$ induces an $\im(\pi_2\pi_1)$-graded surjective homomorphism $\overline{\phi_2\phi_1}\colon (\pi_2\pi_1){_*}S(\mathbf{p}; {\boldsymbol\lambda})\longrightarrow S(\mathbf{q}; {\boldsymbol\mu})_{\im(\pi_2\pi_1)}$, which yields the equivalence (\ref{equiv induced by admissible}) by Proposition \ref{Prop:2.3}.
\end{proof}

\begin{exm}\label{ker=x1-x2} Let $\pi\colon \bbL(2,2,2,2)\rightarrow \bbL(2,2,2,2)$ be a homomorphism defined by:
$$\pi(\vec{x}_1)=\pi(\vec{x}_2)=\vec{d},\;\pi(\vec{x}_3)=\vec{z}_1+\vec{z}_2,\; \pi(\vec{x}_4)=\vec{z}_3+\vec{z}_4.$$
Then $\pi$ is admissible and $\ker\pi=\langle \vec{x}_1-\vec{x}_2\rangle$.

For any $\lambda\in{\bf k}\backslash\{0,1\}$, denote by $\nu=\frac{\sqrt{\lambda}+1}{\sqrt{\lambda}-1}$ and let $\mu=\nu^2$. Then by similar arguments as in the proof of \cite[Proposition 3.8]{CC17}, the following assignments
$$x_1\mapsto -\nu z_1^{2}+z_2^{2},\;x_2\mapsto \nu z_1^{2}+z_2^{2},\; x_3\mapsto 2\sqrt{\nu}z_1z_2,\;x_4\mapsto \sqrt{1-\lambda}z_3z_4$$
define an algebra homomorphism $\phi_{12}: S(2,2,2,2;\lambda)\to S(2,2,2,2;\mu)$, which is obviously compatible with $\pi$ and induces an equivalence
$$({\rm coh}\mbox{-}\mathbb{X}(2,2,2,2;\lambda))^{\langle \vec{x}_1-\vec{x}_2\rangle}\stackrel{\sim}\longrightarrow {\rm coh}\mbox{-}\mathbb{X}(2,2,2,2;\mu).$$
\end{exm}

\subsection{From equivariant equivalences to admissible homomorphisms}

In this subsection we show that equivariant equivalences between the categories of coherent sheaves over weighted projective lines (induced by degree-shift actions) yield admissible homomorphisms between their string groups.

\begin{thm}\label{from equivariant to admissible}
Let $H$ be a finite subgroup of $\mathbb{L}(\bf{p})$. Assume there is an equivariant equivalence $\Phi: ({\rm coh}\mbox{-}\mathbb{X}(\mathbf{p}; {\boldsymbol\lambda}))^{H}
\stackrel{\sim}\longrightarrow
{\rm coh}\mbox{-}\mathbb{X}(\mathbf{q}; {\boldsymbol\mu})$. Then there exists an admissible homomorphism $\pi\colon \mathbb{L}(\mathbf{p})\rightarrow \mathbb{L}(\mathbf{q})$ such that $\ker\pi=H$.
\end{thm}

\begin{proof}
Let $U: ({\rm coh}\mbox{-}\mathbb{X}(\mathbf{p}; {\boldsymbol\lambda}))^{H}
\to {\rm coh}\mbox{-}\mathbb{X}(\mathbf{p}; {\boldsymbol\lambda})$ be the forgetful functor, which admits a left (and right) adjoint functor $F$, known as the induction functor. It follows that both of $F$ and $U$ are exact. Then by \cite[Lemma 2.3.1]{CK},
$$\Ext^{i}_{({\rm coh}\mbox{-}\mathbb{X}(\mathbf{p}; {\boldsymbol\lambda}))^{H}}(FX, Y)\cong \Ext^{i}_{\coh\bbX(\bfp, \bfla)}(X, UY)$$ for any $i\geq 0$, $X\in\coh\bbX(\bfp, \bfla)$ and $Y\in ({\rm coh}\mbox{-}\mathbb{X}(\mathbf{p}; {\boldsymbol\lambda}))^{H}$. In particular, we have $\langle FX, Y\rangle=\langle X, UY\rangle$, where $\langle M, N\rangle= {\rm dim}_{\bf k}\Hom(M,N)-{\rm dim}_{\bf k}\Ext^{1}(M,N)$ is the Euler form.
We prove the theorem in the following steps.

(1) $\Phi\circ F$ is rank-preserved.

Recall that the rank of a coherent sheaf $X$ can be described via Euler form, see for example \cite{BKL}. More precisely, for any ordinary torsion sheaf $S$, we have ${\rm rk}(X)=\langle X, S\rangle$. Observe that $S$ is fixed under the degree-shift action by $H$. Hence $F(S)=(\bigoplus_{\vec{h}\in H}S(\vec{h}), {\rm \varepsilon})=\bigoplus_{\vec{h}\in H}(S, \alpha^{(\vec{h})})$,
where each $(S, \alpha^{(\vec{h})})$ is simple in $({\rm coh}\mbox{-}\mathbb{X}(\mathbf{p}; {\boldsymbol\lambda}))^{H}$ by \cite{RR}. Hence $\Phi((S, \alpha^{(\vec{h})}))$ is an ordinary simple sheaf in ${\rm coh}\mbox{-}\mathbb{X}(\mathbf{q}; {\boldsymbol\mu})$.
Now for any $X\in{\rm coh}\mbox{-}\mathbb{X}(\mathbf{p}; {\boldsymbol\lambda})$,
$$
\langle\Phi\circ F(X), \Phi((S, \alpha^{(\vec{h})}))\rangle
=\langle F(X), (S, \alpha^{(\vec{h})})\rangle=\langle X, U(S, \alpha^{(\vec{h})})\rangle
=\langle X, S\rangle.
$$
That is, ${\rm rk}(\Phi\circ F(X)) ={\rm rk}(X)$, we are done.

(2) $\Phi\circ F$ preserves line bundles. In particular, we can assume $\Phi\circ F$ preserves the structure sheaves.

In fact, for any line bundle $\co(\vec{x})\in\coh\bbX(\bfp,\bfla)$, $F(\co(\vec{x}))=(\bigoplus_{\vec{h}\in H}\co(\vec{x}+\vec{h}), \varepsilon)$. Observe that $\co(\vec{x}+\vec{h}_1)\not\cong\co(\vec{x}+\vec{h}_2)$ for any $\vec{h}_1\neq \vec{h}_2\in H$. Hence $F(\co(\vec{x}))$ is indecomposable, and so does $\Phi\circ F(\co(\vec{x}))$.
According to (1), $\Phi\circ F$ is rank-preserving, hence $\Phi\circ F(\co(\vec{x}))$ is a line bundle.
In particular, up to a degree-shift automorphism of ${\rm coh}\mbox{-}\mathbb{X}(\mathbf{q}; {\boldsymbol\mu})$, we can assume that $\Phi\circ F$ preserves the structure sheaves.

(3) For any simple sheaf $S\in\coh\bbX(\bfp,\bfla)$ with $\tau$-period $r$, the Auslander-Reiten component $\cT_{S}$ of $\coh\bbX(\bfp,\bfla)$ containing $S$ is a non-homogeneous tube of rank $r$. Clearly, the $H$-action on $\coh\bbX(\bfp,\bfla)$ restricts to an $H$-action on $\cT_{S}$. Assume the stablizer subgroup of $S$ has order $m$, then the $H$-orbit of $S$ contains $\frac{n}{m}$-members, where $n=|H|$. By a general result of a group action on a tube (equivalent to the category of nilpotent representations on a cyclic quiver, c.f. \cite{RR}), we obtain that $\Phi\circ F(S)$ splits into $m$-many simples $S_1, S_2,\cdots, S_m$, where each $S_i$ has $\tau$-period $\frac{r}{\frac{n}{m}}=\frac{rm}{n}$. Moreover, $\Phi\circ F$ sends the $r$-tube $\cT_{S}$ to $m$-many $\frac{rm}{n}$-tubes. In particular, if $r=1$, i.e. $\cT_{S}$ is a homogeneous tube, then $m=n$ and $\Phi\circ F$ sends $\cT_{S}$ to $n$-many homogeneous tubes; if $\frac{n}{m}=r$, i.e. the $H$-action on $\cT_{S}$ is transitive, then $\Phi\circ F$ sends the $r$-tube $\cT_{S}$ to $m$-many homogeneous tubes.

(4) For any $1\leq i\leq t$, let $S_i\in\coh\bbX(\bfp,\bfla)$ be an exceptional simple sheaf concentrated in $x_i$. Assume the stablizer subgroup of $S_i$ has order $m_i$ under the $H$-action. Then by (3) we have $\Phi\circ F(S_i)=S_{i_1}\oplus S_{i_2}\oplus \cdots\oplus S_{i_{m_i}}$, where $S_{i_j}$'s are simples belonging to $m_i$-many pairwise distinct orthogonal $\frac{p_im_i}{n}$-tubes. Write $\det(S_{i_j})=\vec{z}_{i_j}$ for $1\leq j\leq m_i$, then $\vec{z}_{i_j}$ is either a generator of $\bbL(\bf q)$ or equal to $\vec{d}$ according to $\frac{p_im_i}{n}>1$ or not. In particular, we have $\frac{p_im_i}{n}\vec{z}_{i_j}=\vec{d}$.

We claim that the assignments $\pi(\vec{x}_i)=\sum\limits_{j=1}^{m_i}\vec{z}_{i_j}$ for $1\leq i\leq t$ linearly extend to a group homomorphism $\pi\colon \bbL(\mathbf{p})\rightarrow \bbL(\mathbf{q})$. In fact, it suffices to show that $\pi(p_i\vec{x}_i)=\pi(p_j\vec{x}_j)$ for any $1\leq i<j\leq t$, which follows from the following equations for any $i$:
$$\pi(p_i\vec{x}_i)=\sum\limits_{j=1}^{m_i}p_i\vec{z}_{i_j}=\sum\limits_{j=1}^{m_i}\frac{n}{m_i}\cdot\frac{p_im_i}{n}\vec{z}_{i_j} =n\vec{d}.$$

(5) $\Phi\circ F(\co(\vec{x}))=\co(\pi(\vec{x}))$ for any $\vec{x}\in \bbL(\bf p)$.

If $\vec{x}=0$, according to (2) we have $\Phi\circ F(\co)=\co=\co(\pi(0))$. Then by induction, it suffices to show that $\Phi\circ F(\co(\vec{x}\pm\vec{x}_i))=\co(\pi(\vec{x}\pm\vec{x}_i))$ provided $\Phi\circ F(\co(\vec{x}))=\co(\pi(\vec{x}))$.

Let $S_i$ be the unique simple sheaf concentrated in $\lambda_i$ fitting into the following exact sequence:
$$0\to \co(\vec{x})\to \co(\vec{x}+\vec{x}_i)\to S_{i}\to 0.$$
By applying the exact functor $\Phi\circ F$, we obtain an exact sequence as follows:
$$0\to \Phi\circ F(\co(\vec{x}))\to \Phi\circ F(\co(\vec{x}+\vec{x}_i))\to \Phi\circ F(S_i)\to 0.$$
By assumption we have $\Phi\circ F(\co(\vec{x}))=\co(\pi(\vec{x}))$, it follows that $\det(\co(\pi(\vec{x})))=\pi(\vec{x})$.
According to (4), $\Phi\circ F(S_i)=\bigoplus\limits_{j=1}^{m_i}S_{i_j}$, hence
$\det(\bigoplus\limits_{j=1}^{m_i}S_{i_j})=\sum\limits_{j=1}^{m_i}\vec{z}_{i_j}=\pi(\vec{x}_i)$. Therefore,
$\det(\Phi\circ F(\co(\vec{x}+\vec{x}_i)))=\pi(\vec{x})+\pi(\vec{x}_i)=\pi(\vec{x}+\vec{x}_i).$
Moreover, by (2) we know that $\Phi\circ F(\co(\vec{x}+\vec{x}_i))$ is a line bundle, hence equal to $\co(\pi(\vec{x}+\vec{x}_i))$.
Similarly, we can prove $\Phi\circ F(\co(\vec{x}-\vec{x}_i))=\co(\pi(\vec{x}-\vec{x}_i))$.

(6) $\ker\pi=H$.

Observe that $\vec{x}\in H$ if and only if $\co$ and $\co(\vec{x})$ belong to the same $H$-orbit, if and only if $F(\co(\vec{x}))=F(\co)$, or equivalently, $\Phi\circ F(\co(\vec{x}))=\Phi\circ F(\co)$. Therefore, according to (5), $\vec{x}\in H$ if and only if $\co(\pi(\vec{x}))=\co$, if and only if $\pi(\vec{x})=0$, i.e. $\vec{x}\in\ker \pi$. Hence $\ker\pi=H$.

(7) $\pi$ is an admissible homomorphism.

Firstly we claim that $\im\pi$ is an effective subgroup of $\bbL(\bf q)$.
Observe that $\im\pi$ is generated by $\pi(\vec{x}_i)=\sum\limits_{j=1}^{m_i}\vec{z}_{i_j},\;1\leq i\leq t$. Then by definition, it suffices to show that $\bbL(\mathbf{q})$ is generated by these $\vec{z}_{i_j}$'s.

In fact, for any generator $\vec{z}_i$ of $\bbL(\bf q)$, let $S'_i$ be an exceptional simple sheaf in $\coh\bbX(\bfq,\bfmu)$ concentrated in $\mu_i$. Then
$\Phi^{-1}(S'_i)$ is simple in $({\rm coh}\mbox{-}\mathbb{X}(\mathbf{p}; {\boldsymbol\lambda}))^{H}$, hence $U(\Phi^{-1}(S'_i))$ is semisimple in ${\rm coh}\mbox{-}\mathbb{X}(\mathbf{p}; {\boldsymbol\lambda})$, where all the simple direct summands are in the same $H$-orbit, in particular, they belong to the same tube of ${\rm coh}\mbox{-}\mathbb{X}(\mathbf{p}; {\boldsymbol\lambda})$. For any indecomposable direct summand $S_i$ of $U(\Phi^{-1}(S'_i))$, we have that $S'_i$ is a direct summand of $\Phi\circ F(S_i)=S_{i_1}\oplus S_{i_2}\oplus \cdots\oplus S_{i_{m_i}}$. Hence $\vec{z}_i=\vec{z}_{i_j}$ for some $1\leq j\leq m_i$. This finishes the proof of the claim.

Secondly, for any $\vec{z}\in \im\pi$, fix a preimage $\vec{x}\in\pi^{-1}(\vec{z})$, then we have
\begin{flalign}
\sum_{\vec{x}'\in \pi^{-1}(\vec{z})}{\rm mult}(\vec{x}')=&\sum_{\vec{x}'\in \pi^{-1}(\vec{z})}\dim_{\bf k}S({\bf p}; {\boldsymbol\lambda})_{\vec{x}'}\nonumber\\
=&\sum_{\vec{h}\in \ker\pi}\dim_{\bf k}S({\bf p}; {\boldsymbol\lambda})_{\vec{x}+\vec{h}}\nonumber\\
=& \dim_{\bf k}\Hom_{\coh\bbX(\bfp, \bfla)}(\co,\bigoplus\limits_{\vec{h}\in H}\co(\vec{x}+\vec{h}))\nonumber\\
=& \dim_{\bf k}\Hom_{\coh\bbX(\bfp, \bfla)}(\co, U\circ F(\co(\vec{x})))\nonumber\\
 =& \dim_{\bf k}\Hom_{\coh\bbX(\bfp, \bfla)^H}(F\co, F(\co(\vec{x})))\nonumber\\
 =& \dim_{\bf k}\Hom_{\coh\bbX(\bfq, \bfmu)}(\Phi\circ F(\co), \Phi\circ F(\co(\vec{x})))\nonumber\\
 =& \dim_{\bf k}\Hom_{\coh\bbX(\bfq, \bfmu)}(\co, \co(\pi(\vec{x})))\nonumber\\
 =&\dim_{\bf k}S({\bf q}, \mu)_{\pi(\vec{x})}\nonumber\\
 =&{\rm mult}(\vec{z}).\nonumber
\end{flalign}
Therefore, $\pi$ is admissile. Then we finish the proof.
\end{proof}

\section{Classification of Equivariant equivalences}

In this section, we aim to classify all the equivariant relations induced by degree-shift actions between the categories of coherent sheaves over weighted projective lines of domestic types and of tubular types respectively.

\subsection{Domestic types}
Let $\bbX(\bf p; {\boldsymbol\lambda})$ be a weighted projective line of domestic type. Then $\mathbf{p}=(),(p), (p_{1}, p_{2})$, $(2,2,p_3)$, $(2,3,3)$, $(2,3,4)$ or $(2,3,5)$ up to permutation. By our normalization assumption, we have $\lambda_1=\infty, \lambda_2=0$ and $\lambda_3=1$. Hence we can simply write $\bbX(\bf p; {\boldsymbol\lambda})=\bbX(\bf p)$.
As an immediate consequence of the Theorem \ref{from admissible to equivariant} and Theorem \ref{from equivariant to admissible}, we have:

\begin{thm} \label{Prop:5.2} Assume $\bbL(\bf p)$ and $\bbL(\bf q)$ are both of domestic types. Let $H$ be a finite subgroup of $\bbL(\bf p)$. Then the following statements are equivalent:
\begin{itemize}
  \item [(1)] there exists an equivalence $({\rm coh}\mbox{-}\mathbb{X}(\mathbf{p}))^{H}\stackrel{\sim}\longrightarrow {\rm coh}\mbox{-}\mathbb{X}(\mathbf{q})$;
  \item [(2)] there exists an admissible homomorphism $\pi\colon \bbL(\mathbf{p})\rightarrow \bbL(\mathbf{q})$ with $\ker\pi=H$.
\end{itemize}
\end{thm}

Therefore, for weighted projective lines of domestic types, we can use the admissible homomorphisms to classify the equivariant relations induced by degree-shift actions between the categories of coherent sheaves, see Figure \ref{fig:admissible for domestic introduction} in the Introduction Section.

\subsection{Tubular types}

Let $\bbX(\bf p; {\boldsymbol\lambda})$ be a weighted projective line of tubular type.
Then $\mathbf{p}=(2,2,2,2)$, $(3,3,3)$, $(4,4,2)$ or $(6,3,2)$ up to permutation. By our normalization assumption, we can write a weighted projective line of type (2,2,2,2) as $\mathbb{X}(2,2,2,2;\lambda)$ with $\lambda\in\mathbf{k}\backslash\{0,1\}$.

In this subsection, we will classify all the equivariant relations induced by degree-shift actions between tubular types, where the parameter $\lambda$ plays a key role. For this we introduce the notation $\Gamma(\lambda)$ to denote the following multiset $$\{\lambda, \frac{1}{\lambda},1-\lambda,\frac{1}{1-\lambda},\frac{\lambda}{\lambda-1}, \frac{\lambda-1}{\lambda}\}.$$
For example,
$\Gamma(-1)=\{-1,-1,2,\frac{1}{2},\frac{1}{2},2\}$
and
$\Gamma(\omega)=\{\omega, -\omega^2, -\omega^2, \omega,-\omega^2, \omega\}$ for $\omega=\frac{1+ \sqrt{-3}}{2}$.
It is easy to see that $\mu\in\Gamma(\lambda)$ if and only if $\Gamma(\lambda)=\Gamma(\mu)$.

\subsubsection{From $(2,2,2,2)$ to $(2,2,2,2)$}

First we consider the equivariant relations for weight type (2,2,2,2) with various parameters.

The following result seems to be well-known to experts, but we can not find a concrete proof in the literature. For the convenience of the reader, we include a proof by using admissible homomorphisms.

\begin{prop}\label{2222ker=0}
There exists an equivalence
$${\rm coh}\mbox{-}\mathbb{X}(2,2,2,2;\lambda)\stackrel{\sim}\longrightarrow {\rm coh}\mbox{-}\mathbb{X}(2,2,2,2;\mu)$$ if and only $\Gamma(\lambda)=\Gamma(\mu)$.
\end{prop}

\begin{proof}
Assume we have an equivalence $\Phi: {\rm coh}\mbox{-}\mathbb{X}(2,2,2,2;\lambda)\stackrel{\sim}\longrightarrow {\rm coh}\mbox{-}\mathbb{X}(2,2,2,2;\mu)$. Then $\Phi$ commutes with the Auslander-Reiten translations on both sides, which are given by the degree-shift of the dualzing elements $\vec{\omega}$ and $\vec{\omega}'$ respectively. Hence, $\Phi\cdot(\vec{\omega})=(\vec{\omega}')\cdot\Phi$. It follows that $\Phi$ induces an equivalence between the equivariant categories: $$\Phi: {\rm coh}\mbox{-}\mathbb{X}(2,2,2,2;\lambda)^{\mathbb{Z}\vec{\omega}}\stackrel{\sim}\longrightarrow {\rm coh}\mbox{-}\mathbb{X}(2,2,2,2;\mu)^{\mathbb{Z}\vec{\omega}'}.$$ Recall that there is an equivalence ${\rm coh}\mbox{-}\mathbb{X}(2,2,2,2;\lambda)^{\mathbb{Z}\vec{\omega}}\cong \coh\mathbb{E}(\lambda)$ for the elliptic curve $\mathbb{E}(\lambda)$ associated to $\lambda$ by \cite[Theorem 7.7]{CCZ}. Thus $\Phi$ induces an equivalence $$ {\rm coh}\mbox{-}\mathbb{E}(\lambda)\stackrel{\sim}\longrightarrow {\rm coh}\mbox{-}\mathbb{E}(\mu).$$ It follows that $\mathbb{E}(\lambda)\cong\mathbb{E}(\mu)$. Hence we get $\Gamma(\lambda)=\Gamma(\mu)$.

On the other hand, assume $\Gamma(\lambda)=\Gamma(\mu)$, i.e. $\mu\in\Gamma(\lambda)$.
For each $\mu$, we define an admissible homomorphism $\pi$ on $\bbL(2,2,2,2)$ and an algebra homomorphism $\phi: S(2,2,2,2;\lambda)\rightarrow S(2,2,2,2;\mu)$ on generators as below.
\begin{table}[ht]
\begin{tabular}{|c|c|c|}
  \hline
  $\mu$&$\pi(\vec{x}_1,\vec{x}_2,\vec{x}_3,\vec{x}_4)$&$\phi(x_1,x_2,x_3,x_4)$\\
  \hline
  $\lambda$ & $(\vec{z}_{1},\vec{z}_{2},\vec{z}_{3},\vec{z}_{4})$& $(z_{1},z_{2},z_{3},z_{4})$\\
  \hline
  $\frac{1}{\lambda}$&$(\vec{z}_{2},\vec{z}_{1},\vec{z}_{3},\vec{z}_{4})$&$(z_{2},z_{1}, \sqrt{-1} z_{3},\sqrt{-\lambda} z_{4})$\\
  \hline
  $1-\lambda$&$(\vec{z}_{4},\vec{z}_{2},\vec{z}_{3},\vec{z}_{1})$&$(z_{4},\sqrt{ \lambda} z_{2}, \sqrt{\lambda-1 }z_{3}, \sqrt{\lambda(1-\lambda)} z_{1})$\\
  \hline
  $\frac{1}{1-\lambda}$&$(\vec{z}_{2},\vec{z}_{3},\vec{z}_{1},\vec{z}_{4})$&$(z_{2},z_{3}, \sqrt{-1} z_{1},\sqrt{1-\lambda}z_{4})$\\
  \hline
  $\frac{\lambda}{\lambda-1}$&$(\vec{z}_{3},\vec{z}_{2},\vec{z}_{1},\vec{z}_{4})$&$(z_{3},z_{2},z_{1},\sqrt{1-\lambda} z_{4})$\\
  \hline
  $\frac{\lambda-1}{\lambda}$&$(\vec{z}_{2},\vec{z}_{4},\vec{z}_{3},\vec{z}_{1})$&$(z_{2},\sqrt{\lambda} z_{4}, \sqrt{\lambda-1} z_{3},\sqrt{1-\lambda} z_{1})$\\
  \hline
  \end{tabular}
\end{table}

\noindent It is easy to see that each $\phi$ in the above table is compatible with $\pi$, and induces a surjective homomorphism between $\im\pi$-graded algebras
$$\bar{\phi}\colon \pi_*S(2,2,2,2;\lambda)\to S(2,2,2,2;\mu)_{\im\pi},$$
which yields an equivalence ${\rm coh}\mbox{-}\mathbb{X}(2,2,2,2;\lambda)\stackrel{\sim}\longrightarrow {\rm coh}\mbox{-}\mathbb{X}(2,2,2,2;\mu)$ by Proposition \ref{Prop:2.3}.
\end{proof}

Let $f(x)=(\frac{x+1}{x-1})^2$ for any $x\neq 1$. It is easy to see that $f(\frac{1}{x})=f(x)$ and $f(-x)=\frac{1}{f(x)}$ for $x\neq \pm 1$. Then for any $\lambda\in \mathbf{k}\setminus\{0,1\}$, we have $f(\sqrt{1-\lambda})=f(\frac{1}{\sqrt{1-\lambda}})=(\frac{\sqrt{1-\lambda}+1}{\sqrt{1-\lambda}-1})^2$ and
$f(\sqrt{\frac{\lambda-1}{\lambda}})=f(\sqrt{\frac{\lambda}{\lambda-1}})=(\frac{\sqrt{\lambda-1}+\sqrt{\lambda}}{\sqrt{\lambda-1}-\sqrt{\lambda}})^2$.
Moreover, $f(\sqrt{f(\sqrt{\lambda})})=\lambda$ or $\frac{1}{\lambda}$, in particular, $\Gamma(f(\sqrt{f(\sqrt{\lambda})}))=\Gamma(\lambda)$.

In the next two propositions we assume $\{i,j,k,l\}=\{1,2,3,4\}$.

\begin{prop}\label{2222ker=2}
There exists an equivalence
$$({\rm coh}\mbox{-}\mathbb{X}(2,2,2,2;\lambda))^{\langle \vec{x}_i-\vec{x}_j\rangle}\stackrel{\sim}\longrightarrow {\rm coh}\mbox{-}\mathbb{X}(2,2,2,2;\mu)$$ if and only if $\Gamma(\mu)=\Gamma(f(\sqrt{\lambda'}))$ for some $\lambda'\in\Gamma(\lambda)$.
\end{prop}

\begin{proof}
Let $\sigma_{ij}$ be a permutation on $\{1,2,3,4\}$ with $\sigma_{ij}(i)=1$ and $\sigma_{ij}(j)=2$. Let $\pi_{\sigma_{ij}}$ be the automorphism on the group $\bbL(2,2,2,2)$ defined by $\pi_{\sigma_{ij}}(\vec{x}_l)=\vec{x}_{\sigma_{ij}(l)}$ for $1\leq l\leq 4$. It is easy to check that $\pi_{\sigma_{ij}}$ is admissible with kernel equal to zero. According to Theorem \ref{from admissible to equivariant} and Proposition \ref{2222ker=0}, there exist some $\lambda'\in\Gamma(\lambda)$ and an algebra homomorphism $\phi_{\sigma_{ij}}: S(2,2,2,2; \lambda)\to S(2,2,2,2; \lambda')$, which is compatible with $\pi_{\sigma_{ij}}$ and induces a surjective homomorphism $\overline{\phi_{\sigma_{ij}}}$, moreover, there is an equivalence $${\rm coh}\mbox{-}\mathbb{X}(2,2,2,2;\lambda)\stackrel{\sim}\longrightarrow {\rm coh}\mbox{-}\mathbb{X}(2,2,2,2;\lambda').$$

Recall from Example \ref{ker=x1-x2} that there exist an admissible homomorphism $\pi_{12}$ on $\bbL(2,2,2,2)$ with $\ker\pi_{12}={\langle \vec{x}_1-\vec{x}_2\rangle}$ and an algebra homomorphism $\phi_{12}: S(2,2,2,2; \lambda')\to S(2,2,2,2; f(\sqrt{\lambda'}))$, which is compatible with $\pi_{12}$ and induces a surjective homomorphism $\overline{\phi_{12}}$ and an equivalence $${\rm coh}\mbox{-}\mathbb{X}(2,2,2,2;\lambda')^{\langle \vec{x}_1-\vec{x}_2\rangle}\stackrel{\sim}\longrightarrow {\rm coh}\mbox{-}\mathbb{X}(2,2,2,2;f(\sqrt{\lambda'})).$$

Now according to Proposition \ref{Prop:3.13}, the composition $\pi_{ij}:=\pi_{12}\pi_{\sigma_{ij}}$ is admissible with $\ker\pi_{ij}=\langle \vec{x}_i-\vec{x}_j\rangle$. Then By Proposition \ref{decomposition of category},
$\overline{\phi_{12}\phi_{\sigma_{ij}}}$ is surjective and induces an equivalence
$${\rm coh}\mbox{-}\mathbb{X}(2,2,2,2;\lambda)^{\langle \vec{x}_i-\vec{x}_j\rangle}\stackrel{\sim}\longrightarrow {\rm coh}\mbox{-}\mathbb{X}(2,2,2,2;f(\sqrt{\lambda'})).$$
Therefore, $({\rm coh}\mbox{-}\mathbb{X}(2,2,2,2;\lambda))^{\langle \vec{x}_i-\vec{x}_j\rangle}\stackrel{\sim}\longrightarrow {\rm coh}\mbox{-}\mathbb{X}(2,2,2,2;\mu)$ if and only if
$$ {\rm coh}\mbox{-}\mathbb{X}(2,2,2,2;\mu)\stackrel{\sim}\longrightarrow {\rm coh}\mbox{-}\mathbb{X}(2,2,2,2;f(\sqrt{\lambda'})),$$ if and only if $\Gamma(\mu)=\Gamma(f(\sqrt{\lambda'}))$. We are done.
\end{proof}

\begin{prop}\label{2222ker=4}
There exists an equivalence
$$({\rm coh}\mbox{-}\mathbb{X}(2,2,2,2;\lambda))^{\langle \vec{x}_i-\vec{x}_j,\vec{x}_i-\vec{x}_k\rangle}\stackrel{\sim}\longrightarrow {\rm coh}\mbox{-}\mathbb{X}(2,2,2,2;\mu)$$ if and only if $\Gamma(\mu)=\Gamma(\lambda)$.
\end{prop}

\begin{proof}
Let $\pi_{\sigma}$ be the admissible homomorphism defined as follows:
$$\pi_{\sigma}: \bbL(2,2,2,2)\to \bbL(2,2,2,2);\quad\vec{x}_i\mapsto \vec{x}_1,\;\vec{x}_j\mapsto \vec{x}_2,\;\vec{x}_k\mapsto \vec{x}_3,\;\vec{x}_l\mapsto \vec{x}_4.$$ Then $\ker\pi_{\sigma}=0$. By Theorem \ref{from admissible to equivariant} and Proposition \ref{2222ker=0},
there exist some $\lambda'\in\Gamma(\lambda)$ and an algebra homomorphism $\phi_{\sigma}: S(2,2,2,2;\lambda)\to (2,2,2,2;\lambda')$, which is compatible with $\pi_{\sigma}$ and induces an equivalence
$${\rm coh}\mbox{-}\mathbb{X}(2,2,2,2;\lambda)\stackrel{\sim}\longrightarrow {\rm coh}\mbox{-}\mathbb{X}(2,2,2,2;\lambda').$$

Let $\pi_{12}$ be the automorphism on $\bbL(2,2,2,2)$ defined by $$\pi_{12}(\vec{x}_1)=\pi_{12}(\vec{x}_2)=\vec{c},\;\pi_{12}(\vec{x}_3)=\vec{x}_1+\vec{x}_2,\;\pi_{12}(\vec{x}_4)=\vec{x}_3+\vec{x}_4.$$
By Example \ref{ker=x1-x2},
there exists an algebra homomorphism $\phi_{12}: S(2,2,2,2; \lambda')\to S(2,2,2,2; f(\sqrt{\lambda'}))$, which is compatible with $\pi_{12}$ and
induces an equivalence $${\rm coh}\mbox{-}\mathbb{X}(2,2,2,2;\lambda')^{\langle \vec{x}_1-\vec{x}_2\rangle}\stackrel{\sim}\longrightarrow {\rm coh}\mbox{-}\mathbb{X}(2,2,2,2;f(\sqrt{\lambda'})).$$
Observe that $\pi_{12}^2$ is admissble with kernel $\langle \vec{x}_1-\vec{x}_2,\vec{x}_1-\vec{x}_3\rangle$. Then by Proposition \ref{decomposition of category}, $\phi_{12}^2$ is compatible with $\pi_{12}^2$ and induces an equivalence
$${\rm coh}\mbox{-}\mathbb{X}(2,2,2,2;\lambda')^{\langle \vec{x}_1-\vec{x}_2,\vec{x}_1-\vec{x}_3\rangle}\stackrel{\sim}\longrightarrow {\rm coh}\mbox{-}\mathbb{X}(2,2,2,2;f(\sqrt{f(\sqrt{\lambda'})})).$$
Similarly, the composition $\pi_{12}^2 \pi_{\sigma}$ is admissble with kernel $\langle \vec{x}_i-\vec{x}_j,\vec{x}_i-\vec{x}_k\rangle$, hence $\phi_{12}^2\phi_{\sigma}$ is compatible with $\pi_{12}^2\pi_{\sigma}$ and induces an equivalence
$${\rm coh}\mbox{-}\mathbb{X}(2,2,2,2;\lambda)^{\langle \vec{x}_i-\vec{x}_j,\vec{x}_i-\vec{x}_k\rangle}\stackrel{\sim}\longrightarrow {\rm coh}\mbox{-}\mathbb{X}(2,2,2,2;f(\sqrt{f(\sqrt{\lambda'})})).$$
Therefore, ${\rm coh}\mbox{-}\mathbb{X}(2,2,2,2;\lambda)^{\langle \vec{x}_i-\vec{x}_j,\vec{x}_i-\vec{x}_k\rangle}\stackrel{\sim}\longrightarrow {\rm coh}\mbox{-}\mathbb{X}(2,2,2,2;\mu)$ if and only if $${\rm coh}\mbox{-}\mathbb{X}(2,2,2,2;\mu)\stackrel{\sim}\longrightarrow {\rm coh}\mbox{-}\mathbb{X}(2,2,2,2;f(\sqrt{f(\sqrt{\lambda'})})),$$ if and only if $\Gamma(\mu)=\Gamma(f(\sqrt{f(\sqrt{\lambda'})}))=\Gamma(\lambda')=\Gamma(\lambda)$.
\end{proof}

\subsubsection{From $(4,4,2)$ to $(2,2,2,2)$}
Now we consider the equivariant relations between weight types (4,4,2) and (2,2,2,2).
For any admissible homomorphism $\pi: \bbL(4,4,2)\rightarrow \bbL(2,2,2,2)$, from Table \ref{table for tubular admissible} in Theorem \ref{classification} we know that $\ker\pi$ has the form $\langle \vec{x}_1-\vec{x}_2\rangle, \langle 2\vec{x}_1-2\vec{x}_2\rangle$ or $\langle 2\vec{x}_1-2\vec{x}_2, 2\vec{x}_1-\vec{x}_3\rangle$.

\begin{prop}\label{442ker=2,4} Assume $H=\langle \vec{x}_1-\vec{x}_2\rangle, \langle 2\vec{x}_1-2\vec{x}_2\rangle$ or $\langle 2\vec{x}_1-2\vec{x}_2, 2\vec{x}_1-\vec{x}_3\rangle$. There exists an equivalence
$$({\rm coh}\mbox{-}\mathbb{X}(4,4,2))^{H}\stackrel{\sim}\longrightarrow {\rm coh}\mbox{-}\mathbb{X}(2,2,2,2;\mu)$$ if and only if $\Gamma(\mu)=\Gamma(-1)$.
\end{prop}

\begin{proof} For each $H$, we claim that there is an equivalence
\begin{equation}\label{equivalence from 442 to 2222}({\rm coh}\mbox{-}\mathbb{X}(4,4,2))^{H}\stackrel{\sim}\longrightarrow {\rm coh}\mbox{-}\mathbb{X}(2,2,2,2;-1).
\end{equation}
Then by Proposition \ref{2222ker=0}, $({\rm coh}\mbox{-}\mathbb{X}(4,4,2))^{H}\stackrel{\sim}\longrightarrow {\rm coh}\mbox{-}\mathbb{X}(2,2,2,2;\mu)$ if and only if $${\rm coh}\mbox{-}\mathbb{X}(2,2,2,2;\mu)\stackrel{\sim}\longrightarrow {\rm coh}\mbox{-}\mathbb{X}(2,2,2,2;-1),$$ if and only if $\Gamma(\mu)=\Gamma(-1).$

For the claim we first consider $H=\langle 2\vec{x}_1-2\vec{x}_2\rangle$.
Observe that the following group homomorphism
$$\pi_1\colon \bbL(4,4,2)\rightarrow \bbL(2,2,2,2);\quad  \vec{x}_1\mapsto \vec{z}_1,\; \vec{x}_2\mapsto \vec{z}_2,\; \vec{x}_3\mapsto \vec{z}_3+ \vec{z}_4$$ is admissible with $\ker\pi_1=\langle 2\vec{x}_1-2\vec{x}_2\rangle$.
According to \cite[Proposition 3.2]{CC17}, there exists an algebra homomorphism $\phi_1: S(4,4,2)\to S(2,2,2,2;-1)$, which is compatible with $\pi_1$ and induces an equivalence (\ref{equivalence from 442 to 2222}).

Secondly we consider $H=\langle \vec{x}_1-\vec{x}_2\rangle$.
Let $\pi_{12}$ be the automorphism on $\bbL(2,2,2,2)$ defined by $$\pi_{12}(\vec{z}_1)=\pi_{12}(\vec{z}_2)=\vec{d},\;\pi_{12}(\vec{z}_3)=\vec{z}_1+\vec{z}_2,\;\pi_{12}(\vec{z}_4)
=\vec{z}_3+\vec{z}_4.$$
Then $\ker\pi_{12}=\langle \vec{z}_1-\vec{z}_2\rangle$. According to Example \ref{ker=x1-x2}, there exists an algebra homomorphism $\phi_{12}: S(2,2,2,2;-1)\to S(2,2,2,2;f(\sqrt{-1}))$, which is compatible with $\pi_{12}$ and
induces an equivalence $${\rm coh}\mbox{-}\mathbb{X}(2,2,2,2;-1)^{\langle \vec{z}_1-\vec{z}_2\rangle}\stackrel{\sim}\longrightarrow {\rm coh}\mbox{-}\mathbb{X}(2,2,2,2;f(\sqrt{-1})).$$
Observe that $\pi_{12}\pi_1$ is admissible with kernel $\langle \vec{x}_1-\vec{x}_2\rangle$ and $f(\sqrt{-1})=-1$. Then by Proposition \ref{decomposition of category}, $\phi_{12}\phi_1$ induces an equivalence (\ref{equivalence from 442 to 2222}).

Finally we consider $H=\langle 2\vec{x}_1-2\vec{x}_2, 2\vec{x}_1-\vec{x}_3\rangle$.
Let $\pi_{34}$ be the automorphism on $\bbL(2,2,2,2)$ defined by $$\pi_{34}(\vec{z}_1)=\vec{z}_1+\vec{z}_2,\;\pi_{34}(\vec{z}_2)
=\vec{z}_3+\vec{z}_4,\;\pi_{34}(\vec{z}_3)=\pi_{34}(\vec{z}_4)=\vec{d}.$$
By similar arguments as in Example \ref{ker=x1-x2}, the following assignments
$$z_1\mapsto (1-\sqrt{-1})z_1z_2,\;z_2\mapsto \sqrt{-1}z_3z_4,\; z_3\mapsto z_1^{2}+\sqrt{-1}z_2^{2},\;z_4\mapsto z_1^{2}-\sqrt{-1}z_2^{2}$$ define an algebra homomorphism $\phi_{34}: S(2,2,2,2;-1)\to S(2,2,2,2;-1)$, which is compatible with $\pi_{34}$ and
induces an equivalence $${\rm coh}\mbox{-}\mathbb{X}(2,2,2,2;-1)^{\langle \vec{z}_3-\vec{z}_4\rangle}\stackrel{\sim}\longrightarrow {\rm coh}\mbox{-}\mathbb{X}(2,2,2,2;-1).$$
Observe that $\pi_{34}\pi_1$ is admissible with kernel $\langle 2\vec{x}_1-2\vec{x}_2, 2\vec{x}_1-\vec{x}_3\rangle$. Then by Proposition \ref{decomposition of category}, $\phi_{34}\phi_1$ induces an equivalence (\ref{equivalence from 442 to 2222}).
\end{proof}

\subsubsection{From $(6,3,2)$ to $(2,2,2,2)$}
Now we consider the equivariant relations between weight types (6,3,2) and (2,2,2,2).
For any admissible homomorphism $\pi: \bbL(6,3,2)\rightarrow \bbL(2,2,2,2)$, from Table \ref{table for tubular admissible} in Theorem \ref{classification} we have $\ker\pi=\langle 2\vec{x}_1-\vec{x}_2 \rangle$.

\begin{prop}\label{632ker}
There exists an equivalence
$$({\rm coh}\mbox{-}\mathbb{X}(6,3,2))^{\langle 2\vec{x}_1-\vec{x}_2 \rangle}\stackrel{\sim}\longrightarrow {\rm coh}\mbox{-}\mathbb{X}(2,2,2,2;\mu)$$ if and only $\Gamma(\mu)=\Gamma(\omega)$.
\end{prop}

\begin{proof}
Recall from \cite[Proposition 3.4]{CC17} that there exists an equivalence
$$({\rm coh}\mbox{-}\mathbb{X}(6,3,2))^{\langle 2\vec{x}_1-\vec{x}_2\rangle}\stackrel{\sim}\longrightarrow {\rm coh}\mbox{-}\mathbb{X}(2,2,2,2;\omega).$$
Therefore, $({\rm coh}\mbox{-}\mathbb{X}(6,3,2))^{\langle 2\vec{x}_1-\vec{x}_2\rangle}\stackrel{\sim}\longrightarrow {\rm coh}\mbox{-}\mathbb{X}(2,2,2,2;\mu)$ if and only if $${\rm coh}\mbox{-}\mathbb{X}(2,2,2,2;\mu)\stackrel{\sim}\longrightarrow {\rm coh}\mbox{-}\mathbb{X}(2,2,2,2; \omega),$$ if and only if $\Gamma(\mu)=\Gamma(\omega).$
\end{proof}

\subsubsection{Main theorem for tubular types}

Now we can classify all the equivariant relations for weighted projective lines of tubuar types.
\begin{thm} \label{classification of tubular case} Assume $\bbL(\mathbf{p})$ and $\bbL(\mathbf{q})$ are both of tubular types. Let $H$ be a finite subgroup of $\bbL(\bf p)$. Then the following statements are equivalent:
\begin{itemize}
  \item [(1)] there exists an equivalence $({\rm coh}\mbox{-}  \mathbb{X}(\mathbf{p}; {\boldsymbol\lambda}))^{H}\stackrel{\sim}\longrightarrow {\rm coh}\mbox{-}  \mathbb{X}(\mathbf{q}; {\boldsymbol\mu})$ for some ${\boldsymbol\lambda}; {\boldsymbol\mu}$;
  \item [(2)] there exists an admissible homomorphism $\pi\colon \bbL(\mathbf{p})\rightarrow \bbL(\mathbf{q})$ with $\ker\pi=H$.
\end{itemize}
 Moreover, all the possibilities for the parameters are listed as below:

\begin{table}[ht]
\begin{tabular}{|c|c|c|c|}
\hline
$(\mathbf{p}; {\boldsymbol\lambda})$&$(\mathbf{q}; {\boldsymbol\mu})$&$H$&$\Gamma(\mu)$\\
\hline
\multirow{2}*{$(2,2,2,2;\lambda$)}&\multirow{2}*{$(2,2,2,2;\mu$)}&$\langle \vec{x}_i-\vec{x}_j\rangle$&$ \Gamma(f(\sqrt{\lambda'})); \;\lambda'\in\Gamma(\lambda)$\\
\cline{3-4}
&&$\langle \vec{x}_i-\vec{x}_j,\vec{x}_i-\vec{x}_k\rangle$&$\Gamma(\lambda)$\\
\hline
\multirow{4}*{$(4,4,2)$}&$(4,4,2)$&$\langle 2\vec{x}_1-\vec{x}_3\rangle,\langle 2\vec{x}_2-\vec{x}_3\rangle$& \\
\cline{2-4}
&\multirow{3}*{$(2,2,2,2;\mu$)}&$\langle \vec{x}_1-\vec{x}_2\rangle$&\multirow{3}*{$\Gamma(-1)$}\\
\cline{3-3}
&&$\langle 2\vec{x}_1-2\vec{x}_2\rangle$&\\
\cline{3-3}
&&$\langle 2\vec{x}_1-2\vec{x}_2,2\vec{x}_1-\vec{x}_3\rangle$&\\
\hline
$(3,3,3)$&$(3,3,3)$&$\langle \vec{x}_i-\vec{x}_j\rangle$&\multirow{2}*{} \\
\cline{1-3}
\multirow{2}*{$(6,3,2)$}&$(3,3,3)$&$\langle 3\vec{x}_1-\vec{x}_3\rangle$& \\
\cline{2-4}
&$(2,2,2,2;\mu$)&$\langle 2\vec{x}_1-\vec{x}_2\rangle$&$\Gamma(\omega)$\\
\hline
\end{tabular}
\vspace{1em}
 \caption{The list of equivariant relations between tubular types}
\end{table}
\end{thm}

\begin{proof} According to Theorem \ref{from admissible to equivariant} and Theorem
\ref{from equivariant to admissible}, we only need to consider the parameters involved. Then the result follows from Propositions \ref{2222ker=2}, \ref{2222ker=4}, \ref{442ker=2,4} and \ref{632ker}.
\end{proof}

\noindent {\bf Acknowledgements.}\quad
The authors are grateful to Xiao-Wu Chen and Helmut Lenzing for their helpful discussions and comments.
S. Ruan is indebted to Henning Krause and William Crawley-Boevey for their supports and hospitalities during his visiting on Bielefeld University since 2018.

This work is supported by the National Natural Science Foundation of China (No.s 11801473, 11871404 and 11971398), the Fundamental Research Funds for Central Universities of China (No.s 20720180002 and 20720180006).

\vskip 5pt
\noindent {\tiny  \noindent Jianmin Chen, Yanan Lin, Shiquan Ruan and Hongxia Zhang\\
School of Mathematical Sciences, \\
Xiamen University, Xiamen, 361005, Fujian, PR China.\\
E-mails: chenjianmin@xmu.edu.cn, ynlin@xmu.edu.cn, sqruan@xmu.edu.cn,
hxzhangxmu@163.com\\ }
\vskip 3pt

\end{document}